\documentclass[12pt]{amsart}

\usepackage[T1]{fontenc}
\usepackage[utf8]{inputenc}
\usepackage[margin=1.05in]{geometry}
\usepackage{amsmath,amssymb,amsthm,mathtools,mathrsfs}
\usepackage{enumitem}
\usepackage{booktabs,array,longtable}
\usepackage{tikz-cd}
\usepackage{tikz}
\usepackage{xcolor}
\usepackage{comment}
\usepackage[unicode, pdfencoding=auto, colorlinks=true,linkcolor=blue,citecolor=blue,urlcolor=blue]{hyperref}
\newcommand{\C}{\mathbb C}
\newcommand{\Z}{\mathbb Z}

\newcommand{\D}{\mathbb D}

\newcommand{\SL}{\operatorname{SL}}
\newcommand{\GL}{\operatorname{GL}}

\newcommand{\Spec}{\operatorname{Spec}}
\newcommand{\rank}{\operatorname{rank}}

\newcommand{\diag}{\operatorname{diag}}

\newcommand{\bi}{\mathbf i}

\newcommand{\cS}{\mathcal S}

\newcommand{\cD}{\mathcal D}

\newcommand{\ol}{\overline}

\newcommand{\eps}{\varepsilon}
\newcommand{\mut}{\operatorname{mut}}

\theoremstyle{plain}
\newtheorem{theorem}{Theorem}[section]
\newtheorem{proposition}[theorem]{Proposition}
\newtheorem{lemma}[theorem]{Lemma}
\newtheorem{corollary}[theorem]{Corollary}
\newtheorem{conjecture}[theorem]{Conjecture}

\theoremstyle{definition}
\newtheorem{definition}[theorem]{Definition}

\newtheorem{example}[theorem]{Example}

\theoremstyle{remark}
\newtheorem{remark}[theorem]{Remark}

\title[On Cluster deep loci in double Bruhat cells]{On Cluster Deep Loci in Double Bruhat Cells}
\author{Quan Hanwen}
\date{Draft compiled \today}

\begin{document}

\begin{abstract}
We study deep loci for the reduced word Berenstein--Fomin--Zelevinsky cluster atlas on type-A double Bruhat cells. As defined in paper \cite{CGSS24}, for a seed collection on a cluster variety $X$, the deep locus is the complement of the union of all corresponding cluster tori. We will focus on double Bruhat cells $G^{u,v}$ for $G=\SL_n$. 

Our main result is several theorems that describe the structure of these deep loci, such as how it act with the Poisson structure, how deep loci of different cells are related and a tracking formula allow us to compute the reduced-word deep locus of a cell from a lower-dimensional one. Also, this paper include explicit low-rank calculations on double Bruhat cells in $\SL_3$ and Borel double Bruhat cells in $\SL_4$. We also discuss the relationship between these results and some general properties of deep loci. 
\end{abstract}

\maketitle
\tableofcontents

\section{Introduction}

Double Bruhat cells were one of the original geometric sources of cluster structures(See \cite{FZ99}).  Let's recall basic definitions.

Let $G$ be a simple complex algebraic group, choose Borel subgroup $B$ and let $B^-$ be the opposite Borel subgroup. The Cartan subgroup is the intersection of $B$ and $B^-$:

$$H=B\cap B^-$$

its normalizer of $H$ in $G$ is denoted by $N_G(H)$, which makes $H$ a normal subgroup of it, and the quotient 

$$W=N_G(H)/H$$

is the Weyl group of $G$. For $u\in W$, we can define a Bruhat cell 

$$BuB:= B\dot{u}B\subset G$$

where $\dot{u}\in N_G(H) \subset G$ is a representative of $u$ in $W$. It's clear that cell does not depend on the choice of representative $\dot{u}$.  Similarly, for $v\in W$, one has the opposite Bruhat cell $B^-\dot{v}B^-$.  Then for any pair $(u,v)\in W\times W$, one has the double Bruhat cell defined as their intersection
\[
   G^{u,v}=B\dot{u}B\cap B^-\dot{v}B^- .
\]

for a given double Weyl group element $(u,v)$. 

Bruhat cells decompose group $G$ into a disjoint union of Bruhat cells $G=\bigsqcup_{u\in W}BuB$. Intersecting Bruhat decompositions for $B$ and $B^-$ we have the decomposition of $G$ into double Bruhat cells:

$$G=\bigsqcup_{u,v\in W} G^{u,v}$$

Fomin and Zelevinsky showed that these varieties admit explicit factorized charts indexed by reduced words of $u$ and $v$, and that the corresponding (generalized) minors give total-positivity tests and explicit inversion formulae for factorized parameters \cite{FZ99}. A double reduced word $\bi$ gives a toric chart
\[
  x_{\bi}:H\times (\C^*)^{\ell(u)+\ell(v)}\dashrightarrow G^{u,v},
  \qquad (a;t_1,\ldots,t_m)\longmapsto a\prod_k x_{i_k}(t_k).d
\]
whose coordinates are the \emph{factorized variables}.  The same variety $G^{u,v}$ also has a second collection of cluster charts, given by certain minors we call $\emph{minor variables}$, and they are related by the Fomin-Zelevinsky(FZ) twist $x'=\zeta^{u,v}(x): G^{u,v}\to G^{u^{-1},v^{-1}}$.The twist maps factorized variables into monomials in generalized minors. 

In the paper \cite{BFZ05}, Berenstein, Fomin, and Zelevinsky give each reduces word a corresponding cluster seed, making the minor variables cluster variables of the cluster algebra structure on the coordinate ring $\C[G^{u,v}]$. A mutation is the elementary replacement of one mutable cluster variable by another according to an exchange relation.  As we will see, those cluster coordinates given by double reduced words are only part of the full cluster atlas of the double Bruhat cell.

The minor variables change as under mutation of cluster variables when we do elementry $2$-moves and $3$-moves, and so do the corresponding monomial of factorized variables under the twist map (we call them {\bf factorized "cluster" variables}).  Thus the two sets of variables on $G^{u,v}$ and $G^{u^{-1},v^{-1}}$ are related by an invertible monomial transformation. 

Berenstein, Fomin, and Zelevinsky identified the coordinate rings of double Bruhat cells with upper cluster algebras defined from these combinatorial data \cite{BFZ05}, which gives $G^{u,v}$ two different cluster variety structures. We use the \emph{original} factorized variables mainly in our paper, which are different from the factorized cluster variables, but their coordinate tori (the \emph{factorization tori}) are the same.

\subsection{Deep loci and atlases}

What is the complement of all cluster tori in a double Bruhat cell? If $X$ is a cluster variety and $T_\Sigma\subset X$ denotes the algebraic torus associated with a seed $\Sigma$, we define the \emph{deep locus}
\[
   \cD(X)=X\setminus \bigcup_\Sigma T_\Sigma.
\]
 as the complement of images of all cluster charts. Equivalently, a point is deep if every seed has at least one vanishing mutable coordinate at that point.  This locus has recently been studied from a general cluster-algebraic and mirror-symmetric viewpoint. Castronovo, Gorsky, Simental, and Speyer formulated a stabilizer-based philosophy for cluster deep loci and verified it in several finite-type and positroid settings \cite{CGSS24}.

There are two distinct atlases in the background.  The first is the \emph{reduced-word BFZ atlas}, whose charts come directly from double reduced words as factorized variables.  The second is the \emph{full mutation atlas}, obtained by allowing all cluster mutations from an initial BFZ seed.  The full mutation atlas is the intrinsic cluster atlas.  The reduced-word atlas is usually smaller but every chart is visible from a double wiring diagram. 

Set $\Sigma_{u,v}$ be all seeds in the reduced-word BFZ atlas, indexed by double reduced words of $(u,v)$, $U_{\sigma \in \sigma}$ be the image of cluster chart of $\sigma$ a double reduced word, we define: 
\[
   \D(G^{u,v})=G^{u,v} \setminus \bigcup_{\Sigma_{u,v}} T_\sigma.
\]

 To keep the statements precise, we write
\[
   \cD_{\cS}(X)=X\setminus\bigcup_{\Sigma\in\cS}T_\Sigma
\]
for the deep locus attached to a specified seed collection $\cS$.  When $\cS$ is the full mutation class, we will get common $\cD(X)=\cD_{\cS}(X)$. In this paper, we say $\D(X)=\cD_{\cS}(X)$ when $\cS$ is the reduced-word BFZ atlas.

In our case where $G$ is of type $A$, all clusters charts are connected to reduced words of Weyl group, and factorizable (cluster) variables are positive rational functions of minors. With every pairs of double reduced words are related by elementry moves, the corresponding seeds are related by finite steps of mutation. 

\subsection{Double Bruhat cells and deep loci as algebraic variety}

By using rank conditions, double Bruhat cells get a natural algebraic variety structure. Generally, it is not affine. As zeros and poles of some regular functions, the cluster deep loci should also be algebraic.

\begin{lemma}[BFZ atlas deep loci is algebraic]
Let $G^{u,v}=X$ be the double Bruhat cell as a cluster variety over $\C$, and let $\cS$ be the set of all factorized cluster charts. For each $\sigma \in \cS$, let $\Sigma$ be the corresponding seed and let
\[
   p_\sigma=\prod_{x\in \Sigma}x
\]
be a monomial of cluster variables, viewed as a regular function on $G^{u,v}$.  Then
\[
   \D(G^{u,v})=Z\bigl(p_\sigma:\sigma\in\cS\bigr).
\]
In particular $\D(X)$ is Zariski closed.
\end{lemma}

This lemma also comes out of considering minor variables pulled back by twist maps.
\bigskip

\subsection{Poisson structure and symplectic leaves of double Bruhat cells}

Let $G$ be equipped with the standard Poisson--Lie structure defined in \cite{STS83} associated with the pair of opposite Borel subgroups $B$ and $B^-$. The induced Poisson structure makes every double Bruhat cell
\[
G^{u,v}=B uB\cap B^- vB^- .
\]
be a family of symplectic leaves indexed by the Cartan subgroup $H$. The symplectic leaves of double Bruhat cells were described by Kogan and Zelevinsky in \cite{KZ02}.

\begin{theorem}[Kogan--Zelevinsky]
For every pair $(u,v)\in W\times W$, there exists a distinguished
symplectic leaf
\[
S^{u,v}\subset G^{u,v}
\]
such that every symplectic leaf of $G^{u,v}$ is obtained by the
action of the Cartan subgroup $H$:
\[
\mathcal L=h S^{u,v},
\qquad h\in H .
\]
\end{theorem}

We found from the structure that the reduced-word deep locus $\D(G^{u,v})$ intersects every symplectic leaf of $G^{u,v}$, so it can never be a Poisson subvariety.

\begin{theorem}
Let $\D(G^{u,v})$ be the reduced-word deep locus. If $\D(G^{u,v})\neq\varnothing$, then $\D(G^{u,v})$ intersects every symplectic leaf of $G^{u,v}$.  Moreover, the natural map
\[\mathcal L_h\simeq \mathcal L_{h'}\]
\[a\mapsto h'h^{-1}a\]
induces an isomorphism between their intersections with $\D(G^{u,v})$.
\end{theorem}

\begin{theorem}
Let
\[
\mathbb D(G^{u,v})=G^{u,v}-\bigcup_{\mathbf i\in R(u,v)}T_{\mathbf i}
\]
be the reduced-word deep locus, then $\mathbb D(G^{u,v})$ is not a Poisson subvariety of
$G^{u,v}$ whenever it is non-empty.
\end{theorem}

We will show the proof of these theorems in section \ref{sec:poisson}.  

\subsection{Symmetries on double Bruhat cells of type \texorpdfstring{$A$}{A} groups}

We now specialize to $G=\SL_n$. The Weyl group is the symmetric group $S_n$.  

In this case, the double Bruhat cells are naturally realized as matrix varieties, and the factorized variables are rational functions of matrix minors.  The reduced-word deep loci are therefore algebraic subvarieties of $\SL_n$ defined by vanishing conditions on factorized cluster variables. We discuss these loci in detail in Section~\ref{sec:typeA}.  

We use three involutive symmetries that relate reduced-word deep loci in double Bruhat cells.  Let

\[
\tau(M)=M^T,\qquad
\alpha(M)=JM^TJ,\qquad
\iota(M)=M^{-1},
\]
where $J$ is the anti-diagonal matrix.

\begin{theorem}[Symmetries of reduced-word deep loci]
\label{thm:intro_symmetries}
For $G=\SL_n$, the three maps above induce biregular identifications
\[
\begin{aligned}
\tau\bigl(\D(G^{u,v})\bigr)
   &=\D(G^{v^{-1},u^{-1}}),\\
\alpha\bigl(\D(G^{u,v})\bigr)
   &=\D(G^{w_0u^{-1}w_0,\,w_0v^{-1}w_0}),\\
\iota\bigl(\D(G^{u,v})\bigr)
   &=\D(G^{u^{-1},v^{-1}}).
\end{aligned}
\]
In particular, all three involutions preserve $\D(G^{w_0,w_0})$.  They also preserve the irreducible components, all finite component intersections, and the dimensions of these intersections.
\end{theorem}

We will refind these observations in the low-rank calculations of Sections~\ref{sec:sl2sl3} and~\ref{sec:sl4borel}.  The explicit actions on double words, factorization parameters of the three involution maps are recorded in Section~\ref{sec:symmetries}.

\subsection{Tracking deep loci through terminal letters}

Our second result gives a recursive method for passing from smaller double Bruhat cells to larger ones.  Let
\[
S_{u,v}=\{i:\ell((u,v)s_i)<\ell(u,v)\}
\]
be the set of possible terminal letters of double reduced words for $(u,v)$.  For $i\in S_{u,v}$, set
\[
\D_i^0=G^{u,v}\setminus
   \bigl(G^{(u,v)s_i}x_i(\C^*)\bigr),
\qquad
\D_i^1=\overline{
   \D(G^{(u,v)s_i})x_i(\C^*)}^{\,G^{u,v}}.
\]

\begin{theorem}[Tracking formula]
\label{thm:intro_tracking}
The reduced-word deep locus satisfies
\[
\D(G^{u,v})
=\bigcap_{i\in S_{u,v}}
  \bigl(\D_i^0\cup\D_i^1\bigr).
\]
Equivalently,
\[
\D(G^{u,v})
=\bigcup_{k:S_{u,v}\to\{0,1\}}
  \bigcap_{i\in S_{u,v}}\D_i^{k(i)}.
\]
Consequently, every irreducible component of $\D(G^{u,v})$ is contained in one of the intersections indexed by a function $k:S_{u,v}\to\{0,1\}$.
\end{theorem}

The tracking formula provides us a recursive method of calculating the deep locus, which will be much simpler than directly compute the complement. Moreover, it provides a systematic way to analyze every irreducible component of a deep locus by partitioning the complement of coordinate tori correspond to words with terminal letter $i$ into $\D_i^0$ or $\D_i^1$. Section~\ref{sec:tracking} will develop this theory, and the later type-$A$ examples in Section~\ref{sec:sl2sl3} and \ref{sec:sl4borel} show how the resulting candidates are traslated into explicit matrix equations.

\subsection{Boundary compatibility}

If $(u,v)\leq(u',v')$ in product Bruhat order, then $G^{u,v}$ is a boundary stratum of $\overline{G^{u',v'}}$.  The natural compatibility question is whether
\[
G^{u,v}\cap\overline{\D(G^{u',v'})}
\subseteq \D(G^{u,v}).
\tag{BC}
\]

\begin{theorem}[Boundary compatibility in weak Bruhat order]
\label{thm:intro_boundary_weak}
The inclusion \textup{\rm(BC)} holds whenever $(u,v)\leq(u',v')$ in the left or right weak Bruhat order.
\end{theorem}

The weak-order hypothesis is essential: the analogous statement for arbitrary comparable elements in product Bruhat order is false.  In $\SL_5$, the pair
\[
\left(
G^{e,\,s_3s_2s_1s_3s_2s_3},
G^{e,\,s_3s_2s_1s_4s_1s_3s_2s_3}
\right)
\]
gives a counterexample: the lower cell contains limits of deep points of the higher cell that are not deep in the lower cell.

We also arrived at a general but weaker codimension-one statement.  A factorization torus of the lower cell is called \emph{extendable} if its word can be obtained by deleting one letter from a reduced word of the higher cell.  Let $\D_{\mathrm{ext}}^{(u,v)\subset(u',v')}$ denote the complement of the union of all such extendable tori.

\begin{proposition}[Boundary compatibility for the extendable atlas]
\label{prop:intro_boundary_extendable}
Suppose
\[
(u,v)<(u',v'),\qquad
\ell(u',v')=\ell(u,v)+1.
\]
Then
\[
G^{u,v}\cap\overline{\D(G^{u',v'})}
\subseteq
\D_{\mathrm{ext}}^{(u,v)\subset(u',v')}.
\]
\end{proposition}

The full discussion and the explicit $\SL_5$ calculation appear in Section~\ref{sec:tracking}.

\subsection{Basic Examples: \texorpdfstring{$\SL_2^{w_0,w_0}$ and $\SL_3^{e,w_0}$}{SL2 and SL3}}

In this section we discuss two basic examples that exhibit the two most important elementary transformations of reduced words. 

The first example, $\SL_2^{w_0,w_0}$, illustrates the \emph{mixed $2$-move} in the Coxeter group $W\times W$. 
The second example, $\SL_3^{e,w_0}$, illustrates the ordinary \emph{braid $3$-move} in the Weyl group $W=S_3$. 
These examples serve as the local models for the coordinate changes of double Bruhat cells.

Throughout, for $\SL_N$ we write
\[
x_i(t)=I+tE_{i,i+1}, \qquad y_i(t)=I+tE_{i+1,i},
\]
and we denote by $w_0$ the longest element of the Weyl group.
\medskip

\paragraph{\bf The cell $\SL_2^{w_0,w_0}$ and mixed $2$-move}: Let $G=\SL_2$, and let $W=S_2=\{e,s_1\}$, so that $w_0=s_1$. 
The open double Bruhat cell is
\[
G^{w_0,w_0}=B w_0 B \cap B^- w_0 B^- 
=\left\{
\begin{pmatrix} a & b \\ c & d \end{pmatrix}\in \SL_2 \;:\; b\neq 0,\; c\neq 0
\right\}.
\]
Since $W\times W$ has two reduced words for $(w_0,w_0)$, namely
\[
\bar{1}1
\qquad\text{and}\qquad
1\bar{1},
\]
this example is the simplest case of mixed $2$-move. Here $\bar{1}$ denotes the generator from the first copy of $W$, while $1$ denotes the generator from the second copy.

\begin{figure}[ht]
\centering
\begin{tikzpicture}[>=stealth, node distance=3.6cm]
\node[draw, rounded corners, inner sep=6pt] (A) {$\bar{1}1$};
\node[draw, rounded corners, inner sep=6pt, right of=A] (B) {$1\bar{1}$};
\draw[->, thick] (A) -- node[above] {mixed $2$-move} (B);
\end{tikzpicture}
\caption{\footnotesize {The two double reduced words for $(w_0,w_0)$ in $\SL_2$, related by the mixed $2$-move.}}
\label{fig:sl2-word-graph}
\end{figure}

\paragraph{\bf The chart attached to $\bar{1}1$ and $1\bar{1}$.}
A point in the corresponding factorization chart of $\bar{1}1$ has the form
\[
x=
\begin{pmatrix}\alpha&0\\0&\alpha^{-1}\end{pmatrix}
\begin{pmatrix}1&0\\ t_1&1\end{pmatrix}
\begin{pmatrix}1&t_2\\0&1\end{pmatrix}
=
\begin{pmatrix}
\alpha & \alpha t_2\\
\alpha^{-1}t_1 & \alpha^{-1}(1+t_1t_2)
\end{pmatrix}.
\]
Hence, on the open subset where this chart is valid, the parameters are recovered as
\[
\alpha=a,\qquad
t_1=ac,\qquad
t_2=\frac{b}{a}.
\]
In particular, the condition $t_1t_2\neq 0$ is equivalent to $b\neq 0$ and $c\neq 0$, while the chart itself also requires $a\neq 0$.

Similarly, the second word gives the factorization
\[
x=
\begin{pmatrix}\alpha'&0\\0&{\alpha'}^{-1}\end{pmatrix}
\begin{pmatrix}1&t_1'\\0&1\end{pmatrix}
\begin{pmatrix}1&0\\ t_2'&1\end{pmatrix}
=
\begin{pmatrix}
\alpha'(1+t_1't_2') & \alpha' t_1'\\
{\alpha'}^{-1} t_2' & {\alpha'}^{-1}
\end{pmatrix}.
\]
Therefore
\[
\alpha'=\frac{1}{d},\qquad
t_1'=bd,\qquad
t_2'=\frac{c}{d},
\]
so this chart is defined on the open subset where $d\neq 0$.On the overlap of the two charts, the two parameter systems are related by
\[
\alpha'=\frac{\alpha}{1+t_1t_2},\qquad
t_1'=t_2(1+t_1t_2),\qquad
t_2'=\frac{t_1}{1+t_1t_2}.
\]
This is the birational coordinate change corresponding exactly to the mixed $2$-move
\[
\bar{1}1 \longleftrightarrow 1\bar{1}.
\]

\paragraph{\bf The deep locus}
From the computation above, we realize the images of charts $U_{\bar{1}1}$ and $U_{1\bar{1}}$ as $$\left\{
\begin{pmatrix} a & b \\ c & d \end{pmatrix}\in \SL_2 \;:\; b\neq 0,\; c\neq 0,\; a\neq 0
\right\}$$ and $$\left\{
\begin{pmatrix} a & b \\ c & d \end{pmatrix}\in \SL_2 \;:\; b\neq 0,\; c\neq 0,\; d\neq 0
\right\}.$$. The complement of union of these $2$ open subsets will be:

$$\left\{
\begin{pmatrix} 0 & b \\ c & 0 \end{pmatrix}\in \SL_2 \;:\; b\neq 0,\; c\neq 0
\right\}.$$

and this is the deep locus of reduce-word BFZ atlas in $\SL_2^{w_0,w_0}$.

\bigskip

\paragraph{\bf The cell $\SL_3^{e,w_0}$ and $3$-move}

We now turn to the first nontrivial $3$-move (braid relation). 
Let $G=\SL_3$ and let $W=S_3=\langle s_1,s_2 \mid s_1^2=s_2^2=e,\; s_1s_2s_1=s_2s_1s_2\rangle$. 
The longest element is
\[
w_0=s_1s_2s_1=s_2s_1s_2.
\]
Since the left Weyl group element is $e$, the cell $G^{e,w_0}$ is the standard open cell in the upper unipotent subgroup, namely
\[
G^{e,w_0}=B\cap B^-w_0B^-.
\]
Its reduced words are the ordinary reduced words for $w_0$, namely
\[
121
\qquad\text{and}\qquad
212,
\]
and these are related by the braid relation.

\begin{figure}[ht]
\centering
\begin{tikzpicture}[>=stealth, node distance=3.6cm]
\node[draw, rounded corners, inner sep=6pt] (A) {$121$};
\node[draw, rounded corners, inner sep=6pt, right of=A] (B) {$212$};
\draw[->, thick] (A) -- node[above] {braid $3$-move} (B);
\end{tikzpicture}
\caption{The two reduced words for $w_0\in S_3$, related by the braid relation $121\leftrightarrow 212$.}
\label{fig:sl3-word-graph}
\end{figure}
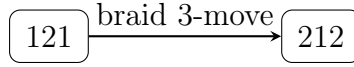

\paragraph{{\bf The chart attached to $121$.}}
The corresponding factorization is
\[
x=a x_1(t_1)x_2(t_2)x_1(t_3)=a
\begin{pmatrix}
1&t_1+t_3&t_1t_2\\
0&1&t_2\\
0&0&1
\end{pmatrix}.
\]
Thus, if we write
\[
x=a
\begin{pmatrix}
1&x_{12}&x_{13}\\
0&1&x_{23}\\
0&0&1
\end{pmatrix},
\]
then in the $121$-chart we have
\[
x_{12}=t_1+t_3,\qquad
x_{23}=t_2,\qquad
x_{13}=t_1t_2.
\]
Assuming $t_2\neq 0$, one may solve for the factorization parameters:
\[
t_2=x_{23},\qquad
t_1=\frac{x_{13}}{x_{23}},\qquad
t_3=x_{12}-\frac{x_{13}}{x_{23}}.
\]
Hence the open subset defined by this chart may be described by $x_{23}\neq 0$. Therefore, $t_1,t_2,t_3$ are all nonzero exactly under the above conditions.

\bigskip

\paragraph{{\bf The chart attached to $212$.}}
Likewise, the second reduced word gives
\[
x=a x_2(s_1)x_1(s_2)x_2(s_3)=a
\begin{pmatrix}
1&s_2&s_2s_3\\
0&1&s_1+s_3\\
0&0&1
\end{pmatrix}.
\]
Therefore
\[
x_{12}=s_2,\qquad
x_{23}=s_1+s_3,\qquad
x_{13}=s_2s_3.
\]
Assuming $s_2\neq 0$, the parameters are
\[
s_2=x_{12},\qquad
s_3=\frac{x_{13}}{x_{12}},\qquad
s_1=x_{23}-\frac{x_{13}}{x_{12}},
\]
and the chart is defined on the open subset $x_{12}\neq 0$

\paragraph{\bf The braid transformation.}
Comparing the two factorizations gives the birational change of coordinates
\[
x_1(t_1)x_2(t_2)x_1(t_3)
=
x_2\!\left(\frac{t_2t_3}{t_1+t_3}\right)
x_1(t_1+t_3)
x_2\!\left(\frac{t_1t_2}{t_1+t_3}\right).
\]
Equivalently,
\[
s_1=\frac{t_2t_3}{t_1+t_3},\qquad
s_2=t_1+t_3,\qquad
s_3=\frac{t_1t_2}{t_1+t_3}.
\]
This is precisely the coordinate change attached to the $3$-move
\[
121\longleftrightarrow 212.
\]
In Coxeter-theoretic language, the two charts correspond to the two reduced expressions of the longest element $w_0\in S_3$, and the birational transition map records the braid relation
\[
s_1s_2s_1=s_2s_1s_2.
\]

\medskip

\paragraph{\bf The deep locus.} From the computation above, the deep locus should consist of matrices with $x_{23}=x_{12}=0$. So we get
$$\D(\SL_3^{e,w_0})=\left\{ 
\begin{pmatrix}
x_{11} & 0 & x_{13}\\
0 & x_{22} & 0\\
0 & 0 & x_{33}
\end{pmatrix}\; : \; x_{11}\neq 0,\; x_{22}\neq 0,\; x_{33}\neq 0,\; x_{13}\neq 0 \right\}
$$

Different from deep locus of $\SL_2^{w_0,w_0}$, this deep locus has non-empty intersection with the infinitesimal neighbourhood of identity matrix.

To summarize, the two basic examples considered here correspond exactly to the two most fundamental word transformations:

\begin{center}
\begin{tabular}{c|c|c}
Example & Reduced-word change & Coxeter interpretation \\
\hline
$\SL_2^{w_0,w_0}$ & $\bar{1}1 \leftrightarrow 1\bar{1}$ & mixed $2$-move in $W\times W$ \\
$\SL_3^{e,w_0}$ & $121 \leftrightarrow 212$ & braid $3$-move in $W=S_3$
\end{tabular}
\end{center}

\subsection{Organization}

Section ~\ref{sec:background} recalls double Bruhat cells, double words, the twist map, and BFZ chamber minors.  Section~\ref{sec:deep} defines seed-collection deep loci and proves the algebraicity and finite-subcollection reductions. Section~\ref{sec:poisson} discusses the Poisson structure of deep loci. Section~\ref{sec:typeA} specializes everything to type $A$ and fixes the matrix-minor notation.  Section~\ref{sec:symmetries} records transpose, anti-diagonal transpose, and the inverse symmetries.  Section~\ref{sec:tracking} gives seed-tracking and boundary-compatibility statements.  Sections~\ref{sec:sl2sl3} and~\ref{sec:sl4borel} contain the explicit $\SL_2$, $\SL_3$, and $\SL_4$ Borel computations.  

\section{Double Bruhat cells and BFZ coordinates}\label{sec:background}

This section fixes the notation used throughout the paper.  We work mostly in type $A$, but it is helpful to recall the general BFZ construction in a form that specializes directly to ordinary minors.

\subsection{Double Bruhat cells}

Let $G$ be a connected simply connected semisimple complex Lie group.  Let $B$ and $B^-$ be opposite Borel subgroups, let $H=B\cap B^-$ be the Cartan subalgebra, and let $W=\operatorname{Norm}_G(H)/H$ be the Weyl group.  For $u,v\in W$ the double Bruhat cell is
\[
   G^{u,v}=BuB\cap B^-vB^-.
\]
The two Bruhat decompositions of $G$ imply
\[
   G=\bigsqcup_{u\in W}BuB
    =\bigsqcup_{v\in W}B^-vB^-,
   \qquad
   G=\bigsqcup_{u,v\in W}G^{u,v}.
\]
The dimension of each double Bruhat cell formula is
\[
   \dim G^{u,v}=\operatorname{rank}G+\ell(u)+\ell(v),
\]
where $\ell$ is the Coxeter length.  In type $A_{n-1}$, $G=SL_n(\C)$ and $W\cong S_n$.

\subsection{Root subgroup coordinates}

Let $s_1,\ldots,s_r$ be the simple reflections in $W$.  Let $x_i(t)$ and $x_{\bar i}(t)$ denote the positive and negative one-parameter root subgroups.  In $SL_n$ they are the elementary matrices
\[
   x_i(t)=I+tE_{i,i+1},
   \qquad
   x_{\bar i}(t)=I+tE_{i+1,i}.
\]
We write barred letters for the first Weyl-group component and unbarred letters for the second component.  Thus a double reduced word for $(u,v)$ is a shuffle of a reduced word for $u$ in barred letters and a reduced word for $v$ in unbarred letters.

\begin{definition}[Double reduced words]
Let $R(u,v)$ be the set of words
\[
   \bi=(i_1,\ldots,i_m),\qquad m=\ell(u)+\ell(v),
\]
in the alphabet $\{\bar1,\ldots,\bar r,1,\ldots,r\}$ whose barred subword is a reduced word for $u$ and whose unbarred subword is a reduced word for $v$.  If $i_k$ is barred or unbarred, then $|i_k|$ denotes the underlying simple index.
\end{definition}

Given $\bi\in R(u,v)$, define the factorization map
\[
   x_{\bi}:H\times (\C^*)^m\longrightarrow G,
   \qquad
   x_{\bi}(a;t_1,\ldots,t_m)=a\,x_{i_1}(t_1)\cdots x_{i_m}(t_m).
\]
A theorem of Fomin--Zelevinsky says that this map is a biregular isomorphism onto a Zariski open subset of $G^{u,v}$ \cite{FZ99}.  The parameters $a,t_1,\ldots,t_m$ are called factorization variables on variety $G^{u,v}$.

\begin{theorem}[Fomin--Zelevinsky factorization chart]\label{thm:fz_factor_chart}
For every $\bi\in R(u,v)$, the map $x_{\bi}$ identifies $H\times(\C^*)^m$ with a Zariski open subset of $G^{u,v}$.  In particular, each double reduced word gives a torus chart on $G^{u,v}$.
\end{theorem}

\subsection{Generalized minors}

For every fundamental weight $\omega_i$ there is a principal generalized minor $\Delta_{\omega_i}$.  If $u,v\in W$, one defines
\[
   \Delta_{u\omega_i,v\omega_i}(x)
      =\Delta_{\omega_i}\bigl(\bar u^{-1}x\bar v\bigr),
\]
where $\bar u,\bar v$ are the standard positive representatives of Weyl group elements, defined by setting $\bar s_i=x_i(-1)x_{\bar i}(1)x_i(-1)$, and $\overline{w w'}=\overline w \overline{w'}$ whenever $\ell(w w')=\ell(w)+\ell(w')$.  The definition is independent of the way we decompose an element into transpositions.  In type $A$ case, this is simply an ordinary matrix minor:
\[
   \Delta_{u\omega_k,v\omega_k}(M)=\Delta_{u[1,k],\,v[1,k]}(M).
\]
Throughout the paper, row and column sets are written in increasing order. $u[1,k]$ means $u$ acts on $\{1,2,...,k\}$ as permutation. 

\subsection{BFZ minor variables attached to a word}

Fix $\bi=(i_1,\ldots,i_m)\in R(u,v)$.  Append the indices
\[
   i_{m+1}=1,
   \quad
   i_{m+2}=2,
   \quad\ldots\quad,
   i_{m+r}=r.
\]
For $1\leq k\leq m+r$, set
\[
 \eps_k=\begin{cases}
 1,& i_k\text{ unbarred or }k>m,\\
 0,& i_k\text{ barred.}
 \end{cases}
\]
The partial Weyl-group products are
\[
   u_{\geq k}=s_{|i_m|}^{1-\eps_m}s_{|i_{m-1}|}^{1-\eps_{m-1}}\cdots s_{|i_k|}^{1-\eps_k},
   \qquad
   v_{<k}=s_{|i_1|}^{\eps_1}s_{|i_2|}^{\eps_2}\cdots s_{|i_{k-1}|}^{\eps_{k-1}}.
\]
For $k>m$ we use the convention $u_{\geq k}=e$ and $v_{<k}=v$.  The minor variables associated with $k$ is
\[
   \Delta_{k,\bi}=\Delta_{u_{\geq k}\omega_{|i_k|},\,v_{<k}\omega_{|i_k|}}.
\]

\begin{remark}
The set of $m+r$ functions $\Delta_{k,\bi}$ is the set of minors cluster variables associated with seed $\bi$.  In type $A$, for $G^{u,v}\subset SL_n$, the frozen boundary minors are the minors $\Delta_{m_j,\bi}=\Delta_{u^{-1}[1,j],[1,j]}\; \textup{and} \; \Delta_{m+j,\bi}=\Delta_{[1,j],v[1,j]}$, and combinatorial data can be read from the double wiring diagram associated with $\bi$. These are precisely the minors that must be inverted to obtain the open double Bruhat cell from the corresponding rank-condition.
\end{remark}

\subsection{The twist map}

The factorization parameters are not themselves the BFZ chamber minors.  The bridge is the twist map.  Let $[g]_- [g]_0 [g]_+$ denote the Gaussian Lower-Diagonal-Upper decomposition, which exists for matrices with non-vanishing leading principal minors.  Fomin and Zelevinsky define a biregular isomorphism
\[
   \zeta^{u,v}:G^{u,v}\longrightarrow G^{u^{-1},v^{-1}}
\]
by
\[
   \zeta^{u,v}(x)=
   \left([\bar u^{-1}x]_-^{-1}\,\bar u^{-1}xv^{-1}\,[xv^{-1}]_+^{-1}\right)^\theta,
\]
where $\theta$ exchanges the two opposite root subgroups and inverts the torus, which gives an algebraic anti-homomorphism.  The inverse of $\zeta^{u,v}$ is $\zeta^{u^{-1},v^{-1}}: G^{u^{-1},v^{-1}}\longrightarrow G^{u,v}$ \cite{FZ99}.

The most important practical consequence is the following monomial relation.

\begin{theorem}[Factorization parameters from twisted minors]\label{thm:twist_monomial}
Let $x\in G^{u,v}$ be in the factorization torus of a word $\bi\in R(u,v)$, and write
\[
   x=a\,x_{i_1}(t_1)\cdots x_{i_m}(t_m).
\]
Set $x'=\zeta^{u,v}(x)$.  Then the parameters $a^{\omega_1},\ldots,a^{\omega_r},t_1,\ldots,t_m$ are Laurent monomials in the chamber minors
\[
   \Delta_{1,\bi}(x'),\ldots,\Delta_{m+r,\bi}(x').
\]
The exponent matrix is invertible over $\Z$ after including the torus characters.
\end{theorem}

This theorem makes sure that factorized variables and minor variables describe the same open torus, even though they look very different. However, only minor variables and its corresponding monomials of factorized variables satisfy rules of mutation when changingcluster seeds.

\subsection{Reduced-word moves}

The reduced-word atlas is connected by local moves.  In a Coxeter group, any two reduced words for the same element are related by braid moves.  For double words there are also mixed commutation moves, which interchange adjacent barred and unbarred letters.  In type $A$, the basic moves are
\[
   i j i \leftrightarrow j i j
   \quad (|i-j|=1),
   \qquad
   i j\leftrightarrow j i
   \quad (|i-j|>1),
\]
inside either the barred or the unbarred alphabet, together with mixed moves
\[
   \bar i j\leftrightarrow j\bar i.
\]
The BFZ exchange relations associated with these moves are identities of positive polynomials of determinants.  For the deep-loci calculations below, the only property we need is that each move replaces one seed torus by an adjacent seed torus and hence contributes another product equation $p_\Sigma=0$ to the deep-locus ideal (in minor variables).

\subsection{The twist as bridge between minors and factorization coordinates}

We now show it explicit, in the two basic examples above, how the Fomin--Zelevinsky twist map
connects the two coordinate systems used throughout the paper.  The convention is the following.
If
\[
X=h\,x_{i_1}(t_1)\cdots x_{i_m}(t_m)\in G^{u,v}
\]
is written in factorization coordinates attached to a double word $\mathbf i$, and if
\[
X^\dagger=\zeta^{u,v}(X)\in G^{u^{-1},v^{-1}}
\]
is its twist, then the factorization parameters of $X$ are Laurent monomials in the chamber
minors of $X^\dagger$.  Thus the twist should be understood as the bridge
\[
\text{factorization coordinates of } X
\quad\longleftrightarrow\quad
\text{chamber-minor coordinates of } X^\dagger .
\]

we will show the result in 2 basic examples: $\SL_2^{w_0,w_0}$ and $\SL_3^{e,w_0}$

\paragraph{\bf \texorpdfstring{$\SL_2^{w_0,w_0}$: the mixed $2$-move}{SL2: the mixed 2-move}}

Let
\[
X=
\begin{pmatrix}
a&b\\
c&d
\end{pmatrix}\in \SL_2^{w_0,w_0},
\qquad
ad-bc=1,
\qquad
b\neq 0,\ c\neq 0 .
\]
For $\SL_2^{w_0,w_0}$ the twist is
\[
X^\dagger=\zeta^{w_0,w_0}(X)
=
\begin{pmatrix}
\dfrac{a}{bc} & \dfrac{1}{c}\\[6pt]
\dfrac{1}{b} & d
\end{pmatrix}.
\]
This formula is useful because the chamber minors of $X^\dagger$ recover the factorization
parameters of $X$ by monomial expressions.

First consider the double word
\[
\mathbf i=\bar{1}1 .
\]
The associated factorization is
\[
X=h(\alpha)y_1(t_1)x_1(t_2),
\qquad
h(\alpha)=
\begin{pmatrix}
\alpha&0\\
0&\alpha^{-1}
\end{pmatrix}.
\]
Multiplying the factors gives
\[
X=
\begin{pmatrix}
\alpha&\alpha t_2\\
\alpha^{-1}t_1&\alpha^{-1}(1+t_1t_2)
\end{pmatrix}.
\]
Therefore
\[
\alpha=a,
\qquad
t_1=ac,
\qquad
t_2=\frac{b}{a}.
\]

Now look at the twisted matrix $X^\dagger$.  For the word $\bar{1}1$, the three relevant
twisted minor variables are
\[
M_1=\Delta_{2,1}(X^\dagger),\qquad
M_2=\Delta_{1,1}(X^\dagger),\qquad
M_3=\Delta_{1,2}(X^\dagger).
\]
Using the formula for $X^\dagger$, we get
\[
M_1=\frac{1}{b},
\qquad
M_2=\frac{a}{bc},
\qquad
M_3=\frac{1}{c}.
\]
Hence
\[
\alpha=\frac{M_2}{M_1M_3},
\qquad
t_1=\frac{M_2}{M_1M_3^2},
\qquad
t_2=\frac{M_3}{M_2}.
\]
Substituting the values of $M_1,M_2,M_3$ gives
\[
\frac{M_2}{M_1M_3}=a,\qquad
\frac{M_2}{M_1M_3^2}=ac,\qquad
\frac{M_3}{M_2}=\frac{b}{a},
\]
exactly as obtained from the direct factorization of $X$.

The other double word is
\[
\mathbf i'=1\bar{1}.
\]
It gives the factorization
\[
X=h(\beta)x_1(r_1)y_1(r_2),
\qquad
h(\beta)=
\begin{pmatrix}
\beta&0\\
0&\beta^{-1}
\end{pmatrix}.
\]
Then
\[
X=
\begin{pmatrix}
\beta(1+r_1r_2)&\beta r_1\\
\beta^{-1}r_2&\beta^{-1}
\end{pmatrix},
\]
so
\[
\beta=\frac{1}{d},
\qquad
r_1=bd,
\qquad
r_2=\frac{c}{d}.
\]
On the twisted side, the relevant minors are
\[
N_1=\Delta_{2,1}(X^\dagger),\qquad
N_2=\Delta_{2,2}(X^\dagger),\qquad
N_3=\Delta_{1,2}(X^\dagger).
\]
Thus
\[
N_1=\frac{1}{b},
\qquad
N_2=d,
\qquad
N_3=\frac{1}{c},
\]
and the factorization parameters are recovered by
\[
\beta=\frac{1}{N_2},
\qquad
r_1=\frac{N_2}{N_1},
\qquad
r_2=\frac{1}{N_2N_3}.
\]

The two factorization charts are related by the mixed $2$-move
\[
\bar{1}1 \longleftrightarrow 1\bar{1}.
\]
Indeed, comparing
\[
h(\alpha)y_1(t_1)x_1(t_2)
=
h(\beta)x_1(r_1)y_1(r_2)
\]
gives
\[
\beta=\frac{\alpha}{1+t_1t_2},
\qquad
r_1=t_2(1+t_1t_2),
\qquad
r_2=\frac{t_1}{1+t_1t_2}.
\]
In minor coordinates on the twisted side, the same move is encoded by the exchange
relation
\[
M_2N_2=1+M_1M_3.
\]
This implies 
Thus the rank-one mixed Coxeter move has two equivalent forms:

\[
\boxed{
\bar{1}1\longleftrightarrow 1\bar{1}
}
\qquad
\Longleftrightarrow
\qquad
\boxed{
M_2N_2=1+M_1M_3 .
}
\]

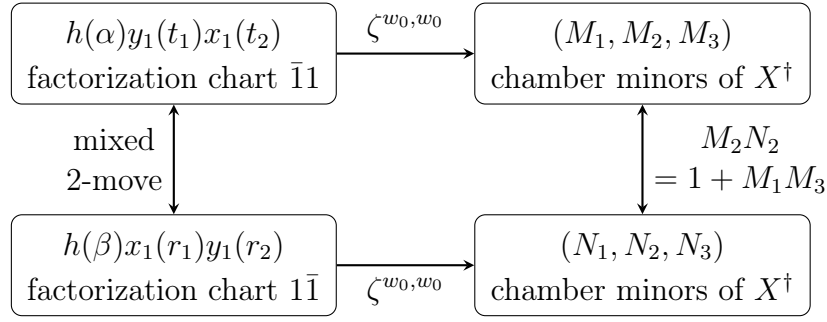
\begin{figure}[ht]
\centering
\begin{tikzpicture}[>=stealth, scale=1]
\node[draw, rounded corners, align=center, inner sep=6pt] (A) at (0,1.4)
{$h(\alpha)y_1(t_1)x_1(t_2)$\\[2pt] factorization chart $\bar{1}1$};
\node[draw, rounded corners, align=center, inner sep=6pt] (B) at (6.2,1.4)
{$(M_1,M_2,M_3)$\\[2pt] chamber minors of $X^\dagger$};
\node[draw, rounded corners, align=center, inner sep=6pt] (C) at (0,-1.4)
{$h(\beta)x_1(r_1)y_1(r_2)$\\[2pt] factorization chart $1\bar{1}$};
\node[draw, rounded corners, align=center, inner sep=6pt] (D) at (6.2,-1.4)
{$(N_1,N_2,N_3)$\\[2pt] chamber minors of $X^\dagger$};

\draw[->, thick] (A) -- node[above] {$\zeta^{w_0,w_0}$} (B);
\draw[->, thick] (C) -- node[below] {$\zeta^{w_0,w_0}$} (D);
\draw[<->, thick] (A) -- node[left, align=center] {mixed\\$2$-move} (C);
\draw[<->, thick] (B) -- node[right, align=center] {$M_2N_2$\\$=1+M_1M_3$} (D);
\end{tikzpicture}
\caption{\footnotesize {The twist identifies factorization coordinates of $X$ with chamber-minor coordinates of $X^\dagger$.  In the simplest case, the mixed $2$-move becomes the exchange relation $M_2N_2=1+M_1M_3$.}}
\label{fig:sl2-twist-factorization-chambers}
\end{figure}

\section{Cluster deep loci on altases}\label{sec:deep}

This section gives the ideal-theoretic formulation of cluster deep loci. We are trying to generate the original definition of deep loci in \cite{CGSS24} and describe the deep locus we focus on.

\subsection{Seed collections and mutable products}

Let $X=\Spec A$ be an affine variety over $\C$.  Suppose we have a collection $\cS$ of cluster seeds (cluster altas) on $X$, with all frozen variables inverted in $A$.  For a seed $\Sigma\in\cS$, write
\[
   \Sigma_{\mut}=\{x_1,\ldots,x_N\}
\]
for the mutable cluster variables in that seed.  Its cluster torus is the principal open subset
\[
   T_\Sigma=\textup{Im}(\C^{*N} \to X) \subset X.
\]
The frozen variables do not appear in $p_\Sigma$ because they are already units in $A$, the type $A$ cases will be proved later.

\begin{definition}[Deep loci on atlases]\label{def:seed_deep}
The deep locus of $X$ with respect to $\cS$ is
\[
   \cD_{\cS}(X)=X\setminus\bigcup_{\Sigma\in\cS}T_\Sigma.
\]
If $\cS$ is the full mutation class, then $\cD_{\cS}(X)$ is the cluster deep locus and is denoted $\cD(X)$.  In this paper we write $\D(X)$ for  $D_\cS$ when $\cS$ is the reduced-word BFZ atlas.
\end{definition}

\begin{proposition}[Monotonicity in the seed collection]\label{prop:monotonicity}
If $\cS\subseteq \cS'$, then
\[
   \cD_{\cS'}(X)\subseteq \cD_{\cS}(X).
\]
Thus a calculation using only reduced-word seeds gives a closed superset of the full cluster deep locus.
\end{proposition}

\begin{proof}
The union of tori for $\cS'$ contains the union of tori for $\cS$.  Taking complements reverses inclusion.
\end{proof}

\begin{lemma}[BFZ atlas deep loci is algebraic]\label{thm:intro_closed}
Let $G^{u,v}=X$ be the double Bruhat cell as a cluster variety over $\C$, and let $\cS$ be the set of all factorized cluster charts. For each $\sigma \in \cS$, let $\Sigma$ be the corresponding seed and let
\[
   p_\sigma=\prod_{x\in \Sigma}x
\]
be a monomial of cluster variables, viewed as a regular function on $G^{u,v}$.  Then
\[
   \D(G^{u,v})=Z\bigl(p_\sigma:\sigma\in\cS\bigr).
\]
In particular $\D(X)$ is Zariski closed.
\end{lemma}

\begin{proof}
We start with a similar version for minor coordinates. Since minor variables are regular on $G^{u,v}$, the factorized cluster variable should be regular as a pull-back. We then have $T_\Sigma=D(p_\Sigma)$ and its complement is $V_X(p_\Sigma)$.  Therefore
\[
 X\setminus\bigcup_{\Sigma\in\cS}T_\Sigma
 =\bigcap_{\Sigma\in\cS}V_X(p_\Sigma)
 =V_X(p_\Sigma:\Sigma\in\cS).
\]
This proves the result for both factorized cluster variables and minor variables.
\end{proof}

\subsection{Matrix-coordinate computation}

Let $X=G^{u,v}\subset SL_n$.  In practice we compute inside a polynomial ring
\[
   R=\C[x_{ij}:1\leq i,j\leq n, \textup{det}(x_{ij}) = 1 ].
\]
The double Bruhat cell is an open subvariety of a closed subvariety generated by some rank conditions (which will be introduced in \ref{sec:typeA}). 
\medskip
 
Let $J_{u,v}\subset R$ be an ideal of rank conditions, and let $q_{u,v}$ be the product of frozen boundary minors that are inverted on $G^{u,v}$.  Given a seed collection $\cS$, choose polynomial representatives of the mutable products $p_\Sigma$.  The algebraic description of the deep locus is
\[
   I_{\cS}^{\mathrm{mat}}
   =\left(\sqrt{J_{u,v}+Z(p_\Sigma:\Sigma\in\cS)}:q_{u,v}^{\infty}\right).
\]
This ideal defines the same closed subset inside $G^{u,v}$ as $I_\cS\subset\C[G^{u,v}]$. The saturation removes components contained entirely in the complement of the open double Bruhat cell.  For example, in the open cell $SL_4^{w_0,w_0}$ the six frozen boundary minors are
\[
   \Delta_{4,1},\quad
   \Delta_{34,12},\quad
   \Delta_{234,123},\quad
   \Delta_{1,4},\quad
   \Delta_{12,34},\quad
   \Delta_{123,234}.
\]
Any component on which one of these minors vanishes is outside $G^{w_0,w_0}$.

\section{The Poisson structure of deep loci}\label{sec:poisson}

A Poisson structure on a manifold $X$ is a bivector field $\pi$ which defines a Lie bracket on smooth functions by
\[
\{f,g\}=\pi(df,dg),
\]
satisfying the Jacobi identity and the Leibniz rule. We call a manifold with Poisson structure a Poisson manifold. Map between Poisson manifolds is called \emph{Poisson} if it keeps the Poisson structure.  

We define \emph{symplectic leaves} for a Poisson manifold be its maximal integral
submanifolds of the distribution
\[
\operatorname{Im}(\pi^\sharp:T^*X\rightarrow TX),
\]
where the induced Poisson structure becomes a symplectic structure.

A Poisson--Lie group is a Lie group $G$ equipped with a Poisson structure $\pi$ such that the multiplication map
\[
m:(G\times G, \pi_1+\pi_2)\rightarrow (G,\pi), \qquad (g_1,g_2)\mapsto g_1g_2,
\]
is Poisson. Equivalently, the Poisson bivector should satisfy
\[
\pi(gh)
=
(L_g)_*\pi(h)+(R_h)_*\pi(g),
\]
for any $g,h\in G$. For $G$ a connected simply connected semisimple complex Lie group, $B$ and $B^-$ be a pair of opposite Borel subgroups, we may equip it with the standard Poisson--Lie structure. 

The paper \cite{KZ02} proved that under this Poisson structure, a double Bruhat cell will become disjoint union of a family of symplectic leaves indexed by the Cartan subgroup. We start from this structure, and prove that the deep loci are not Poisson submanifolds.

\subsection{Poisson--Lie structure from opposite Borel subgroups}

We will construct the standard Poisson--Lie structure from opposite Borel subgroups for group
\[
G=SL_n(\mathbb C).
\]

Consider the corresponding Lie algebra of $G$:
\[
\mathfrak g=\mathfrak{sl}_n(\mathbb C)
\]
with a triangular decomposition
\[
\mathfrak g=\mathfrak n_-\oplus\mathfrak h\oplus\mathfrak n_+,
\]
where $$\mathfrak b=\mathfrak h\oplus\mathfrak n_+,\qquad \mathfrak b^-=\mathfrak h\oplus\mathfrak n_-$$ are the Lie algebras
of opposite Borel subgroups $B$ and $B^-$.

Choose the standard basis
\[
E_{ij},\qquad 1\leq i,j\leq n,
\]
of $\mathfrak{sl}_n$, we have the classical $r$-matrix
\[
r=\frac12\sum_{i=1}^{n-1}H_i\otimes H_i+\sum_{i<j}E_{ij}\otimes E_{ji},
\]
where $H_i$ is the standard basis of the Cartan subalgebra. The associated Poisson bivector is obtained by left and right translations of $r$ from the origin:
\[
\pi(g)=(L_g)_*r-(R_g)_*r .
\]
In particular, the Poisson structure is compatible with the Bruhat decomposition and restricts naturally to double Bruhat cells
\[
G^{u,v}=B uB\cap B^- vB^- ,\qquad u,v\in W,
\]
which inherit the induced Poisson structure from $G$, each consist of symplectic leaves indexed by the Cartan subgroup $H$. The symplectic leaves of double Bruhat cells were described by Kogan and Zelevinsky in \cite{KZ02}. For every pair $(u,v)$ they found a distinguished symplectic leaf
\[
S^{u,v}\subset G^{u,v},
\]
such that every symplectic leaf in $G^{u,v}$ is obtained from
$S^{u,v}$ by the action of the Cartan subgroup:
\[
\mathcal L_h=S^{u,v}\cdot h,
\qquad h\in H .
\]

\subsection{Symplectic leaves and factorization tori}

Fix a double reduced word
\[
\mathbf i=(i_1,\ldots,i_m)\in R(u,v),
\]
and let
\[
T_{\mathbf i}
=
x_{\mathbf i}
(H\times(\mathbb C^*)^m)
\subset G^{u,v}
\]
be the corresponding factorization torus.

The following proposition describes the intersection between the Poisson symplectic foliation and the reduced-word atlas.

\begin{proposition}
\label{prop:leaf_factorization}
Let $\mathcal L$ be a symplectic leaf of $G^{u,v}$ and let
$\mathbf i\in R(u,v)$. Then
\[
\mathcal L\cap T_{\mathbf i}
\]
is a non-empty Zariski open subset of $\mathcal L$. Moreover, the natural map
\[\mathcal L_h\simeq \mathcal L_{h'}\]
\[a\mapsto ah^{-1}h'\]
induces an isomorphism between their intersections with $T_{\mathbf i}$.
\end{proposition}

\begin{proof}
By the description of symplectic leaves in \cite{KZ02}, it is enough to consider the
distinguished leaf $S^{u,v}$, since all other leaves are obtained by
the Cartan right action, which also act freely on the coordinate torus.
\end{proof}

\begin{remark}
   The original version of the description of symplectic leaves in \cite{KZ02} is given by right actions on the distinguished leaf. However, a left action version will make it more clear, since our factorization tori are defined by Cartan left actions.
\end{remark}

\subsection{Symplectic leaves and reduced-word deep loci}

The reduced-word BFZ deep locus
\[
\mathbb D(G^{u,v})=G^{u,v}-\bigcup_{\mathbf i\in R(u,v)}T_{\mathbf i}
\]
is the complement of the union of all factorization tori. We now show that the reduced-word deep locus will intersect every symplectic leaf, as long as it is non-empty. 

\begin{theorem}
\label{thm:deep_intersects_leaves}
Let $\D(G^{u,v})$ be the reduced-word deep locus. If $\D(G^{u,v})\neq\varnothing$, then $\D(G^{u,v})$ intersects every symplectic leaf of $G^{u,v}$.  Moreover, the natural map
\[\mathcal L_h\simeq \mathcal L_{h'}\]
\[a\mapsto ah^{-1}h'\]
induces an isomorphism between their intersections with $\D(G^{u,v})$.
\end{theorem}

\begin{proof}
   \(\D(G^{u,v})\) is the complement of the union of all factorization tori. By Proposition~\ref{prop:leaf_factorization}, every symplectic leaf intersects every factorization torus in a non-empty Zariski open subset. Therefore, if \(\D(G^{u,v})\) is non-empty, it must intersect every symplectic leaf.
\end{proof}

The following lemma:

\begin{lemma}
   \label{lem:closed_contain_leaves}
Let $(X,\pi)$ be a smooth Poisson manifold and $Y\subset X$ be a
closed Poisson submanifold. Then every symplectic leaf of $X$ intersecting
$Y$ is contained in $Y$.
\end{lemma}

tells us that the reduced-word deep locus cannot be a Poisson subvariety, for otherwise it would contain all the double Bruhat cell. This induce the final theorem of this section.

\begin{theorem}[Deep loci are not Poisson submanifolds]
\label{thm:deep_not_poisson}
Let
\[
\mathbb D(G^{u,v})=G^{u,v}-\bigcup_{\mathbf i\in R(u,v)}T_{\mathbf i}
\]
be the reduced-word deep locus, then $\mathbb D(G^{u,v})$ is not a Poisson subvariety of
$G^{u,v}$ whenever it is non-empty.
\end{theorem}

\begin{proof}
Choose
\(
p\in\mathbb D(G^{u,v})
\)
and let $\mathcal L_p$ be the symplectic leaf through $p$. Assume that $\mathbb D(G^{u,v})$ is a Poisson subvariety. Then by the previous lemma,
\[
\mathcal L_p\subseteq\mathbb D(G^{u,v}).
\]

Choose any double reduced word $\mathbf i\in R(u,v)$, by Proposition~\ref{prop:leaf_factorization},
\[
\mathcal L_p\cap T_{\mathbf i}\neq\varnothing .
\]

which contradicts
\[
\mathcal L_p\subseteq\mathbb D(G^{u,v}).
\]

Hence $\mathbb D(G^{u,v})$ cannot be a Poisson subvariety.
\end{proof}

\noindent We end this section with a proof of the lemma \ref{lem:closed_contain_leaves}:
\begin{proof}
Let $\mathcal{L}$ be a symplectic leaf of $X$ and let $y\in\mathcal{L}\cap Y$. Since $Y$ is a Poisson submanifold, the Poisson tensors of $X$ and $Y$ have the same characteristic distribution along $Y$. The symplectic leaf of $Y$ through $y$ is an open subset of $\mathcal{L}$, therefore $\mathcal{L}\cap Y$ is open in $\mathcal{L}$. 

Notice that $Y$ is closed in $X$, the intersection $\mathcal{L}\cap Y$ is also closed in $\mathcal{L}$. As $\mathcal{L}$ is connected and $\mathcal{L}\cap Y\neq\varnothing$, it follows that $\mathcal{L}\cap Y$ should be $\mathcal{L}$, which means $\mathcal{L}\subset Y$.
\end{proof}

\section{Matrix coordinates and rank conditions for \texorpdfstring{$G$}{G} of type A}\label{sec:typeA}

We now specialize to the $G=\SL_n(\C)$, and then discuss how the structures above related to more computable objects, such as minors, ranks or matrix indices.

\subsection{Permutation}

Let $W=S_n$ act on $\{1,\ldots,n\}$.  The simple reflection $s_i$ is the transposition $(i,i+1)$.  We identify $w\in S_n$ with its permutation matrix, using the convention that the nonzero entry in column $j$ lies in row $w(j)$.  Thus the longest element $w_0$ sends $i$ to $n+1-i$.

For a subset $I\subset[1,n]$, let $wI=\{w(i):i\in I\}$, always written in increasing order when it is used as a row or column set of a minor.  For $|I|=|J|$, we write
\[
   \Delta_{I,J}(M)=\det M_{I,J}
\]
for the minor with row set $I$ and column set $J$.  We suppress braces in small examples: for instance,
\[
   \Delta_{34,12}=\Delta_{\{3,4\},\{1,2\}}.
\]

With these conventions,
\[
   \Delta_{u\omega_k,v\omega_k}=\Delta_{u[1,k],v[1,k]}.
\]

\subsection{Type-A root subgroups}

For $1\leq i\leq n-1$,
\[
   x_i(t)=I+tE_{i,i+1},
   \qquad
   x_{\bar i}(t)=I+tE_{i+1,i}.
\]
The diagonal torus is
\[
   H=\{\diag(h_1,\ldots,h_n):h_i\in\C^*,\ h_1\cdots h_n=1\}.
\]
For a double word $\bi=(i_1,\ldots,i_m)$, the factorization chart is
\[
   x_\bi(h;t_1,\ldots,t_m)
      =h\,x_{i_1}(t_1)\cdots x_{i_m}(t_m).
\]
In type $A$, every coordinate of $x_\bi$ is a polynomial in the factorization parameters.  Conversely, on the corresponding open chart, each factorization parameter is a Laurent monomial in minor variables of the twisted point.

\subsection{Rank conditions for Bruhat cells}

The following rank conditions are a convenient way to write equations for cell closures.  For a matrix $M$, let $M^*_{i,j}$ be the northeast $i\times j$ submatrix, and let ${}_* M_{[i,n],[1,j]}$ be the southwest $i\otimes j$ submatrix.

For $w\in S_n$, define
\[
   r^w_{ij}=\#\{k\leq i: w(k)> n-j
\}.
\]
Then $M\in \ol{BwB}$ if and only if
\[
   \rank {}_* M_{i,j}\leq r^w_{ji}
   \qquad (1\leq i,j\leq n).
\]
The open cell $BwB$ is obtained by imposing the appropriate equalities and nonvanishing conditions.  Similarly, $M\in\ol{B^-wB^-}$ can be written using southwest rank conditions.  One convenient form is
\[
   \rank M_{i,j}^*\leq r^{w^{-1}}_{ij},
\]

The double Bruhat cell closure is the intersection of the two rank-condition closures, together with $\det M=1$.

\begin{example}[$\SL_3$ upper Borel]
For $u=e$, the condition $M\in BeB=B$ says exactly that $M$ is upper triangular.  For $v=w_0$ there are no additional vanishing equations from the opposite closure, but the open cell requires the opposite boundary minors
\[
   \Delta_{1,3}\neq0,
   \qquad
   \Delta_{12,23}\neq0.
\]
On the upper unipotent matrix
\[
   U_3(x,y,z)=
   \begin{pmatrix}
   1&x&z\\
   0&1&y\\
   0&0&1
   \end{pmatrix},
\]
these become
\[
   z\neq0,
   \qquad
   xy-z\neq0.
\]
\end{example}

\subsection{Borel double Bruhat cells}

The calculations in Sections~\ref{sec:sl2sl3} and~\ref{sec:sl4borel} focus on
\[
   G^{e,w_0}=B\cap B^-w_0B^-.
\]
Since the diagonal torus acts by multiplication and all diagonal entries are units, the factorized variables keeps invariant and vanishing of minors are controlled by the upper unipotent part.  Thus it is natural to write
\[
   B=H\ltimes N,
\]
and to compute cells in $N$ first.  Multiplying the final answer by $H$ gives the corresponding subset of $G^{e,w_0}$.

For $SL_4$ we use the upper unipotent coordinates
\[
U_4(a,b,c,d,e,f)=
\begin{pmatrix}
1&a&b&c\\
0&1&d&e\\
0&0&1&f\\
0&0&0&1
\end{pmatrix}.
\]
The three opposite boundary minors are
\begin{align*}
   \Delta_{1,4}(U_4)&=c,\\
   \Delta_{12,34}(U_4)&=be-cd,\\
   \Delta_{123,234}(U_4)&=adf-ae-bf+c.
\end{align*}
The open cell condition is
\[
   c(be-cd)(adf-ae-bf+c)\neq0.
\]

\subsection{Minor variables for \texorpdfstring{$G^{e,w_0}$}{G(e,w0)}}

When $u=e$, the double words are just reduced words for $v$, written in the unbarred alphabet.  For $v=w_0$, a reduced word $\bi=(i_1,\ldots,i_N)$ gives minor variables
\[
   \Delta_{[1,i_k],\,v_{<k}[1,i_k]}
   \qquad (1\leq k\leq N+n-1),
\]
where $v_{<k}=s_{i_1}\cdots s_{i_{k-1}}$ for $k\leq N$ and $v_{<k}=w_0$ for the appended frozen indices.  The principal minors are $1$ on the upper unipotent quotient; the nontrivial frozen minors are the opposite boundary minors.  The remaining minors are mutable chamber variables.

\section{Three symmetries of type-A deep loci}\label{sec:symmetries}

Instead of determining each deep locus independently, we study the relations among deep loci of different double Bruhat cells.

We first consider relations between cells of the same dimension, which arise from some involutive symmetries of $\SL_n$. We then study relations between cells of different dimensions: the lower-to-higher direction leads to the \emph{tracking formula}, while the higher-to-lower direction leads to the boundary compatibility problem.

In this section, we focus on three involutive symmetries: the ordinary transpose, the anti-diagonal transpose, and matrix inversion. They relate double Bruhat cells and factorization tori of different cells,  and therefore induce constraints of reduced-word deep loci.

\subsection{The ordinary transpose}

Let
\[
   \tau:\SL_n\longrightarrow \SL_n,
   \qquad
   \tau(M)=M^T.
\]
Then $\tau$ is an involution and an anti-homomorphism:
\[
   \tau(MN)=\tau(N)\tau(M).
\]
It interchanges the two Borel subgroups:
\[
   \tau(B)=B^- ,
   \qquad
   \tau(B^-)=B.
\]
For a permutation matrix $\dot w$, one has $\tau(\dot w)=\dot w^{-1}$.  Hence
\[
   \tau(G^{u,v})=G^{v^{-1},u^{-1}}.
\]
On minors,
\[
   \tau^*(\Delta_{I,J})=\Delta_{I,J}\circ\tau=\Delta_{J,I}.
\]
On elementary root subgroups,
\[
   \tau(x_i(t))=x_{\bar i}(t),
   \qquad
   \tau(x_{\bar i}(t))=x_i(t).
\]
Because products are reversed, a double word
\[
   \bi=(i_1,\ldots,i_m)\in R(u,v)
\]
is sent to
\[
   \bi^\tau=(\tau(i_m),\ldots,\tau(i_1))
   \in R(v^{-1},u^{-1}),
\]
where $\tau(i)=\bar i$ and $\tau(\bar i)=i$.  Commuting the diagonal factor back to the left only rescales the factorization parameters by nonzero torus characters.  Therefore
\[
   \tau(T_{\bi})=T_{\bi^\tau}.
\]

\begin{proposition}[Transpose symmetry of deep loci]\label{prop:transpose_symmetry}
Let $\cS$ be a reduced-word seed collection on $G^{u,v}$ and let $\tau(\cS)$ be the transported seed collection on $G^{v^{-1},u^{-1}}$.  Then
\[
   \tau\bigl(\cD_{\cS}(G^{u,v})\bigr)
   =\cD_{\tau(\cS)}(G^{v^{-1},u^{-1}}).
\]
In particular, for the open cell $G^{w_0,w_0}$, transpose preserves the reduced-word deep locus.
\end{proposition}

\begin{proof}
The transpose identifies every factorization torus $T_{\bi}$ with the transported torus $T_{\bi^\tau}$.  It therefore identifies the corresponding unions of tori.  Taking complements gives the result.
\end{proof}

\subsection{The anti-diagonal transpose}

Let $J$ be the anti-diagonal permutation matrix, $Je_i=e_{n+1-i}$, and define
\[
   \alpha(M)=J M^T J.
\]
This is reflection of the matrix across the anti-diagonal:
\[
   \alpha(M)_{ij}=M_{n+1-j,n+1-i}.
\]
Again $\alpha$ is an involutive anti-homomorphism.  Unlike the ordinary transpose, it preserves the two Borel subgroups:
\[
   \alpha(B)=B,
   \qquad
   \alpha(B^-)=B^-.
\]
For $w\in S_n$,
\[
   \alpha(\dot w)=\dot w_0\dot w^{-1}\dot w_0.
\]
Thus
\[
   \alpha(G^{u,v})
   =G^{w_0u^{-1}w_0,\,w_0v^{-1}w_0}.
\]
For a subset $I\subset[1,n]$, write
\[
   I^\vee=w_0(I)=\{n+1-i:i\in I\}.
\]
Then
\[
   \alpha^*(\Delta_{I,J})=\Delta_{J^\vee,I^\vee}.
\]
There is no sign in the vanishing calculation, since both row and column orders are reversed.

On elementary root subgroups,
\[
   \alpha(x_i(t))=x_{n-i}(t),
   \qquad
   \alpha(x_{\bar i}(t))=x_{\overline{n-i}}(t).
\]
Therefore a double word $\bi=(i_1,\ldots,i_m)$ is transported to
\[
   \bi^\alpha=(\alpha(i_m),\ldots,\alpha(i_1)),
\]
where $\alpha(i)=n-i$ and $\alpha(\bar i)=\overline{n-i}$.  As above, the diagonal factor only produces nonzero rescalings, and hence
\[
   \alpha(T_{\bi})=T_{\bi^\alpha}.
\]

\begin{proposition}[Anti-diagonal symmetry of deep loci]\label{prop:anti_diag_symmetry}
Let $\cS$ be a reduced-word seed collection on $G^{u,v}$ and let $\alpha(\cS)$ be its transported seed collection on $G^{w_0u^{-1}w_0,w_0v^{-1}w_0}$.  Then
\[
   \alpha\bigl(\cD_{\cS}(G^{u,v})\bigr)
   =\cD_{\alpha(\cS)}
   \bigl(G^{w_0u^{-1}w_0,w_0v^{-1}w_0}\bigr).
\]
In particular, $\alpha$ preserves the reduced-word deep loci of $G^{w_0,w_0}$ and $G^{e,w_0}$.
\end{proposition}

\begin{proof}
The anti-diagonal transpose identifies $T_{\bi}$ with $T_{\bi^\alpha}$.  Taking complements of the transported unions of factorization tori proves the statement.
\end{proof}

\subsection{Inversion}\label{subsec:inversion_symmetry}

Let
\[
   \iota:\SL_n\longrightarrow\SL_n,
   \qquad
   \iota(M)=M^{-1}.
\]
Since $\det(M)=1$, the entries of $M^{-1}$ are polynomial cofactors, so $\iota$ is an algebraic involution of $\SL_n$.  It is also an anti-homomorphism:
\[
   \iota(MN)=\iota(N)\iota(M).
\]
Inversion preserves both Borel subgroups.  Therefore
\[
   (B\dot uB)^{-1}=B\dot u^{-1}B,
   \qquad
   (B^-\dot vB^-)^{-1}=B^-\dot v^{-1}B^-,
\]
and hence
\[
   \iota(G^{u,v})=G^{u^{-1},v^{-1}}.
\]

Set
\[
   \beta_i=\alpha_i,
   \qquad
   \beta_{\bar i}=-\alpha_i.
\]
we can express the commuting relation between $x_\eta$($\eta \in [n]\sqcup [n]$) and $h\in H$ as
\[
   h x_\eta(t)=x_\eta\bigl(h^{\beta_\eta}t\bigr)h
\]

As a result, for a double reduced word
\[
   \bi=(i_1,\ldots,i_m)\in R(u,v),
\]
we define the reversed double word
\[
   \bi^\iota=(i_m,\ldots,i_1).
\]
The barred and unbarred subwords are both reversed, so
\[
   \bi^\iota\in R(u^{-1},v^{-1}).
\]
For an element $x\in T_{\bi}$, we have
\[
   x=x_{\bi}(h;t_1,\ldots,t_m)
   =h x_{i_1}(t_1)\cdots x_{i_m}(t_m),
\]
and the inverse $x^{-1}$ lies in $T_{\bi^\iota}$ since:
\begin{align*}
   x^{-1}
   &=x_{i_m}(-t_m)\cdots x_{i_1}(-t_1)h^{-1}\\
   &=h^{-1}
     x_{i_m}\bigl(-h^{\beta_{i_m}}t_m\bigr)
     \cdots
     x_{i_1}\bigl(-h^{\beta_{i_1}}t_1\bigr).
\end{align*}

We finally arrived at a biregular map induced by inversion map $\iota$:

\[
T_{\bi} \simeq H \times \C^{*m} \longrightarrow H \times \C^{*m} \simeq T_{\bi^\iota},
\]
\[
   (h;t_1,\ldots,t_m)
   \longmapsto
   \bigl(h^{-1};-h^{\beta_{i_m}}t_m,\ldots,
   -h^{\beta_{i_1}}t_1\bigr)
\]
between $T_{\bi}$ and $T_{\bi^\iota}$. Consequently,
\[
   \iota(T_{\bi})=T_{\bi^\iota}.
\]

\begin{proposition}[Inversion symmetry of reduced-word deep loci]
\label{prop:inversion_symmetry}
Let $\cS$ be a collection of reduced-word factorization tori on $G^{u,v}$, and let $\iota(\cS)$ be the collection obtained by reversing the corresponding double words.  Then
\[
   \iota\bigl(\cD_{\cS}(G^{u,v})\bigr)
   =\cD_{\iota(\cS)}(G^{u^{-1},v^{-1}}).
\]
In particular, for the reduced-word BFZ atlases containing all double reduced words,
\[
   \iota\bigl(\cD(G^{u,v})\bigr)
   =\cD(G^{u^{-1},v^{-1}}).
\]
\end{proposition}

\begin{proof}
Word reversal is a bijection
\[
   R(u,v)\longrightarrow R(u^{-1},v^{-1}),
   \qquad
   \bi\longmapsto\bi^\iota.
\]
Since $\iota(T_{\bi})=T_{\bi^\iota}$ for every double reduced word,
\[
   \iota\left(
      \bigcup_{\bi\in R(u,v)}T_{\bi}
   \right)
   =
   \bigcup_{\mathbf j\in R(u^{-1},v^{-1})}T_{\mathbf j}.
\]
Since the inversion map is a bijection from $G^{u,v}$ to $G^{u^{-1},v^{-1}}$.  Taking complements on each side proves the result.
\end{proof}

\begin{corollary}[Inversion and component intersections]
\label{cor:inversion_components}
Suppose
\[
   \cD(G^{u,v})=C_1\cup\cdots\cup C_r
\]
is the irreducible decomposition of the reduced-word deep locus.  Then
\[
   \cD(G^{u^{-1},v^{-1}})
   =\iota(C_1)\cup\cdots\cup\iota(C_r)
\]
is its irreducible decomposition.  Moreover, for every subset $A\subseteq\{1,\ldots,r\}$,
\[
   \iota\left(\bigcap_{a\in A}C_a\right)
   =\bigcap_{a\in A}\iota(C_a).
\]
In particular, corresponding finite intersections have the same nonemptiness and, whenever nonempty, the same dimension.
\end{corollary}

\begin{proof}
Inversion is a biregular isomorphism, so it preserves irreducibility, maximal irreducible subsets, finite intersections, and dimensions.
\end{proof}

\subsection{Consequences for computations}

For the open double Bruhat cell $G^{w_0,w_0}$ in $\SL_n$, the maps $\tau$, $\alpha$, and $\iota$ are automorphisms.  Any proposed list of deep-locus components must therefore be stable under the group generated by these involutions.  On minors, their actions are
\[
   \Delta_{I,J}\longmapsto \Delta_{J,I},
   \qquad
   \Delta_{I,J}\longmapsto \Delta_{J^\vee,I^\vee},
   \qquad
   \Delta_{I,J}\longmapsto
   \sum \pm\Delta_{J^c,I^c}.
\]
The sign in the last formula is irrelevant for vanishing.

More generally, inversion gives a direct isomorphism between the reduced-word atlases on $G^{u,v}$ and $G^{u^{-1},v^{-1}}$.  Thus the relation between their deep loci is stronger than equality of component-intersection data: the two deep loci are biregularly isomorphic, and the entire incidence pattern of their irreducible components is preserved.

\section{Tracking deep loci through words and boundary compatibility}\label{sec:tracking}
In this section, we introduce a tracking method for the reduced-word deep loci. The key observation is that the reduced-word atlas admits a natural decomposition according to the terminal letters of double reduced words. This allows us to trace irreducible components of deep loci from smaller double Bruhat cells to larger ones. Also, we discuss boundarial behaviour of Bruhat boundary stratification. 

In the following sections we compute deep loci of all cells in $\SL_3$, and in Borel subgroup of $\SL_4$, which gives us examples as application of these ideas.

Throughout this section, we write
\[
\D^{u,v}=\D(G^{u,v})
\]
for the deep locus associated with the reduced-word BFZ atlas.

\subsection{Tracking through terminal letters}

Let
\[
\mathbf i=(i_1,\ldots,i_m)\in R(u,v)
\]
be a double reduced word. We call $i_m$ the terminal letter of
$\mathbf i$. The possible terminal letters are precisely the simple
reflections which decrease the Coxeter length:
\[
S_{u,v}
=
\{i:\ell((u,v)s_i)<\ell(u,v)\}.
\]
Here $i$ can be either a barred or an unbarred letter, corresponding
to the two Weyl group factors.

From now on we discuss mainly for deep loci of BFZ reduced-word atlases. Recall we use the notation $\D(G^{u,v})=\D^{u,v}$ for deep loci of these altases for double Bruhat cells. Usually the deep locus of a double Bruhat cell decomposes into different irreducible deep components, for example:

\[
\D(\SL_3^{w_0,w_0})=\{ \left( \begin{matrix}
0 & a & b\\
0 & c & d\\
e & 0 & 0 
\end{matrix}\right) \} \cup \{ \left(\begin{matrix}
a & 0 & b\\
0 & c & 0\\
e & 0 & d 
\end{matrix}\right)\}\cup \{\left(\begin{matrix}
0 & 0 & b\\
a & c & 0\\
e & d & 0 
\end{matrix}\right)\}
\]

where $b,c,e \neq 0$ for each component.

\medskip

The following observation explains how deep components can be tracked along the reduced-word graph through their terminal letters. We divide the BFZ word atlas into different groups by terminal Coxeter letters. For every $i\in S_{u,v}$, deleting the terminal letter gives a
bijection
\[
\{
\mathbf i\in R(u,v):\mathbf i\text{ ends with }s_i
\}
\longleftrightarrow
R((u,v)s_i).
\]
Hence the reduced-word atlas decomposes into families indexed by
terminal letters. 

\begin{theorem}[Decomposition by terminal letters]
Let $G=\SL_n$. Given double reduced word $(u,v)\in W\times W$ of length $l$, and $S_{u,v}$ be set of all possible terminal letters of its double reduced word. Then we have:

\[
\D^{u,v}=G^{u,v}-\bigcup_{S_{u,v}} U_i=\bigcap_{S_{u,v}} (G^{u,v}\setminus U_i)
\]

where $U_i$ be union of coordinate tori of all double reduced words of $(u,v)$ with terminal letter $i$.

\end{theorem}
\begin{proof}
   Notice that
\[
\D^{u,v}=G^{u,v}- \bigcup_{\bi\in R(u,v)} T_\bi =G^{u,v}-\bigcup_{S_{u,v}} U_i
\]
\end{proof}

The part of the atlas with fixed terminal letter $i$ is controlled by the product chart
\[
G^{(u,v)s_i}\times x_i(\mathbb C^*)
\longrightarrow
G^{u,v}.
\]

A point escapes all tori in the family when 

\begin{itemize}
   
   \item either its smaller-cell component is already deep in $G^{(u,v)s_i}$, 
   \item or the point lies outside the product $G^{(u,v)s_i}\times x_i(\mathbb C^*)$. 
   
\end{itemize}

Let $\D^{0}_i$ be the complement of $G^{(u,v)s_i}x_i(\C^*)$ in $G^{u,v}$ ,$\D^1_i$ be closure of $\D^{(u,v)s_i}x_i(\C^*)$ inside the whole cell. Notice that $\D^{u,v}$ is always a closed subset of $G^{u,v}$, therefore $\D^{(u,v)s_i}x_i(\C^*)$ should be a closed subset in the open subset $G^{u,v}-\D_i^{0}$, which means $\D^1_i$ is $\D^{(u,v)s_i}x_i(\C^*)$ with limit points on $\D_i^{0}$.

\medskip

Also, Each irreducible component of $\D^{u,v}$ lies in $\D^{0}_i$ or $\D^{1}_i$(or maybe both) for each $i\in S_{u,v}$, which gives a map $k: S_{u,v} \to \{0,1\}$. Therefore, we can find out how one irreducible component of $\D^{u,v}$ shows up inductively, and that's what ``tracking'' means.

\begin{theorem}

Let $G=\SL_n$. Given double reduced word $(u,v)\in W\times W$ of length $l$, and $S_{u,v}$ be set of all possible terminal letters of its double reduced word. With $\D_i^{0}$ and $\D_i^{1}$ defined as above, we have:

\[
\D^{u,v}=\bigcap_{S_{u,v}} \D^{0}_i\cup \D^{1}_i.
\]

Moreover, for each irreducible component of the deep locus, there exists a $k: S_{u,v} \to \{0,1\}$, such that

\[
\D^{u,v}_k=\bigcap_{S_{u,v}} \D^{k(i)}_i
\]

cover the irreducible component.
\end{theorem}

This theorem admits us to track each irreducible component of deep locus by every possible suffix to the deep locus of a smaller cell or $\D^{0}_i$. Since the deep loci of Cartan subgroup is empty, each irreducible component should come out of intersections of right multiplication by $x_i(\C^*)$'s on some $\D^{0}_i$ of smaller cells.

\subsection{Proof of the tracking theorem}

\begin{proposition}

The set $S_{u,v}$ contains all generators of double Weyl group, the right multiplication of which would decrease the inversion number of $(u,v)$. That is to say:

\[
S_{u,v}=\{i\in [n]\sqcup [n]\; |\; \ell((u,v)s_i)<\ell (u,v)\}
\]
\end{proposition}

Since all coordinate torus of BFZ reduced-word are related to a double reduced word of $(u,v)$, they are naturally divided into different groups by their trailing character. We say:

\[
\begin{aligned}
\D(G^{u,v})&=G^{u,v}-\bigcup_{(u,v)=s_{i_1}...s_{i_l}}Hx_{i_1}(\C^*)x_{i_2}(\C^*)...x_{i_l}(\C^*)\\
&=G^{u,v}-\bigcup_{i\in S_{u,v}}\bigcup_{(u,v)=s_{i_1}s_{i_2}...s_{i}}Hx_{i_1}(\C^*)x_{i_2}(\C^*)...x_{i}(\C^*)
\end{aligned}
\]

Now we have the key observation: The union of images of all coordinate tori corresponding to words with prefix $i$ is given by right multiplication of $x_i(\C^*)$ on union of all coordinate tori of double reduced words of $(u,v)s_i$, which is naturally the complement of deep locus in $G^{(u,v)s_i}$. 

Consider the map:

\[
\C^* \times G^{(u,v)s_i}\to G^{u,v}\;,\; (t,g) \mapsto g \cdot x_i(t)
\]

The map is injective: we can recover $t$ and $g$ by cancelling the rank of
northeast/southwest submatrices using elementary column transformations.
This gives a bijection between the complement of coordinate tori of words
with terminal letter $i$ and $\D^{0}_i\cup \D^{1}_i$.
Now we get:

\[
\begin{aligned}
\D(G^{u,v})&=G^{u,v}-\bigcup_{i\in S_{u,v}}\bigcup_{(u,v)=s_{i_1}s_{i_2}...s_{i}}Hx_{i_1}(\C^*)x_{i_2}(\C^*)...x_{i}(\C^*)\\
&=G^{u,v}-\bigcup_{i\in S_{u,v}}(G^{(u,v)s_i}-\D(G^{(u,v)s_i}))\cdot x_{i}(\C^*)\\
&=\bigcap_{i\in S_{u,v}}(G^{u,v}-(G^{(u,v)s_i}\cdot x_{i}(\C^*)-\D(G^{(u,v)s_i})\cdot x_{i}(\C^*)))\\
&=\bigcap_{i\in S_{u,v}}(G^{u,v}-G^{(u,v)s_i}\cdot x_{i}(\C^*))\cup \D(G^{(u,v)s_i})\cdot x_{i}(\C^*)
\end{aligned}
\]

Which is just $\D^{u,v}=\bigcap_{S_{u,v}} \D^{0}_i\cup \D^{1}_i$. Notice that $\D^{0}$ and ${\D^{1}}$ are closed subsets of the cell. Then, for each irreducible component of $\D^{u,v}$, it should  intersect with at most one of  $\D^{0}_i-{\D^1_i}$ and $\D^1_{i}-\D^0_{i}$. This defines a map  $k: S_{u,v} \to \{0,1\}$ we want.

\subsection{Bruhat order and closure of double Bruhat cells, Boundary compatibility}

The double Bruhat cell stratification is naturally related to Bruhat order in double Weyl Group. The Bruhat order on a coxeter group $W$ may be characterized by the subword property: $u\leq w$ if and only if some reduced word for $w$ contains a reduced word for $u$ as a subword.  On pairs we also have
\[
   (u,v)\leq(u',v')
   \quad\Longleftrightarrow\quad
   u\leq u'\text{ and }v\leq v'.
\]
The closure relation for double Bruhat cells is
\[
   \ol{G^{u',v'}}=\bigcup_{(u,v)\leq(u',v')}G^{u,v}.
\]
we know that lower cells are boundary strata of higher cells. The compactiblility of deepness will be formulated as the following:

\begin{conjecture}[Bruhat-boundary compatibility]\label{conj:boundary}
For some pair of double Weyl group elements $(u,v)\leq(u',v')$, we have
\[
   G^{u,v}\cap \ol{\D(G^{u',v'})}
   \subseteq
   \D(G^{u,v}).
\]
Equivalently, a point that is non-deep in the boundary cell should not be a limit of deep points from the larger cell. We will give examples of word pairs making this conjecture true or false in the following sections.
\end{conjecture}

\begin{theorem}
The conjecture holds when $(u,v)<(u',v')$ in the \emph{(left/right) weak Bruhat order}, which means

$$(u',v')=s_{i_1}s_{i_2}...s_{i_k}(u,v)\;\textup{(or }(u,v)s_{i_1}s_{i_2}...s_{i_k}\textup{)}, \textup{ for } \ell(u',v')=\ell(u,v)+k$$
\end{theorem}

\begin{proof}

In this case, reduced words of $(u,v)$ has a natural embedding to the set of reduced word of $(u',v')$, by adding $i_1,...i_k$ on the left(right). 

Notice that every torus in the higher cell could extend to a torus in the lower cell by admit one factorized parameter achieve the value $0$, we can achieve a lemma:

\begin{lemma}
Let $G=\SL_n$. For $x\in G^{u,v}$ in some torus $T_{\bi}$, there exist a neighbourhood $G\supset G^{u,v}\sqcup G^{s_{i_1}...s_{i_n}(u,v)}\supset U\ni x$ for any $s_{i_k}$ such that
$$U\cap G^{s_{i_1}...s_{i_n}(u,v)}\subset T_{i\sqcup \bi},\;\;U \cap G^{u,v} \subset T_{\bi} $$
as open subsets, whenever $\ell(s_{i_1}...s_{i_n}(u,v))= n+ \ell(u,v)$.
\end{lemma}

\begin{proof}
First we prove this theorem for $n=1$. For any word $\bi=j_1...j_{\ell(u,v)}$, we define the map
$$x_{\bi,i\sqcup \bi}\C \times \C^{* \ell(u)+\ell(v)+n-1}\to \overline{G^{u,v}}$$
$$(t;a_1,a_2,...a_{n-1},t_1,t_2,...t_{\ell(u,v)}) \mapsto h(a_1,...a_{n-1})\cdot x_i(t)\prod x_{j_k}(t_k)$$
which satisfies $x_{\bi,i\sqcup}|_{0\times  \C^{* \ell(u)+\ell(v)+n-1}}=x_{\bi}$, and $x_{\bi,i\sqcup}|_{C^{*}\times  \C^{* \ell(u)+\ell(v)+n-1}} = x_{i\sqcup\bi}$
then this map is a bijection to its image, then by invariance of domain it's an open map, so the image will be open.

For general $n$, we can construct the openset inductively by constructing subset for a serie of going-up cells $U_k\subset  G^{s_{i_{n-k+1}}...s_{i_n}(u,v)}$, satisfying $U_0$ be neighbourhood of $x$, and $U_k \sqcup U_{k+1}$ open in the union of this two cells. Then $U_0\sqcup U_n= \bigcup U_k\cap G^{u,v}\sqcup G^{s_{i_1}...s_{i_n}(u,v)} $ is an opensubset.
\end{proof}

Then each non-deep point of the lower cell contains in a open set covered by coordinate tori on higher cells, which gives the boundary compatibility of this case.

\end{proof}
The conjecture is wrong for general pair of bulk and boundary cells.  A counter-example is given by pair $$(\SL_5^{e,s_3s_2s_1s_3s_2s_3}, \SL_5^{e,s_3s_2s_1s_4s_1s_3s_2s_3})$$ the deep locus of the latter cell has elements:

\[\left\{
\begin{pmatrix}
1 & l & m &qm & 0\\
0 & 1 & s & 0 & 0\\
0 & 0 & 1& q+r & tr\\
0 & 0 & 0 & 1 & t\\
0 & 0 & 0 & 0 & 1
\end{pmatrix},l,m,q,s,t,r \neq 0\right\}
\]

take all possible limit points in the cell correspond to word $321323$, we get

\[\left\{
\begin{pmatrix}
1 & l & m &qm & 0\\
0 & 1 & s & 0 & 0\\
0 & 0 & 1& q & 0\\
0 & 0 & 0 & 1 & 0\\
0 & 0 & 0 & 0 & 1
\end{pmatrix},l,m,q,s
 \neq 0\right\}
\]

which has points outside the deep locus of this cell.

\subsection{A weaker version of boundary compatibility}
However, if we ignore all non-extendable coordinate tori on the lower cell, then a weak version of boundary compatibility holds for $(u',v')<(u,v)$ with $\ell(u',v')=\ell(u,v)-1$. A coordinate torus $T_\bi$ of $(u',v')$ is called extendable if there exists a coordinate torus $T_{\bi'}$ of $(u,v)$ such that $\bi'$ is obtained from $\bi$ by deleting one letter. Whether this weaker version of boundary compatibility holds for general pair of cells is still open.

\begin{proposition}
   For a pair of double Weyl group elements $(u,v)$ and $(u',v')$ with
\[
(u',v') < (u,v),\qquad \ell(u',v')=\ell(u,v)-1
\]
one has
\[
G^{u',v'}
\cap
\overline{\D(G^{u,v})}
\subseteq
\D_{\mathrm{ext}}^{(u',v')\subset(u,v)}.
\]
where $\D_{\mathrm{ext}}^{(u',v')\subset(u,v)}$ is the complement of all extendable coordinate tori of $(u',v')$ in $G^{u',v'}$.

\end{proposition}

\begin{proof}
Suppose that
\[
x\in
G^{u',v'}
\setminus
\D_{\mathrm{ext}}^{(u',v')\subset(u,v)}
\cap
\overline{\D(G^{u,v})}
\]
Then there exists an extendable word
\(
\mathbf j\in R_{\mathrm{ext}}^{(u,v)}(u',v')
\)
such that
\(
x\in T_{\mathbf j}.
\)

Choose an extension of $\mathbf j$
\[
\mathbf i=(i_1,\ldots,i_{m+1})\in R(u,v)
\]
with
\[
\mathbf j=(i_1,\ldots\hat{i_k},\ldots,i_{m+1}).
\]
 Allow the
$k$-th factorization parameter to vanish and consider
\[
\widetilde{x}_{\mathbf i,\mathbf j}:H\times(\mathbb C^*)^{k-1}\times\mathbb C\times(\mathbb C^*)^{m+1-k}\longrightarrow G^{u',v'}\sqcup G^{u,v},
\]
defined by
\[
\widetilde{x}_{\mathbf i,\mathbf j}
(h;t_1,\ldots,t_{m+1})
=
h x_{i_1}(t_1)\cdots x_{i_{m+1}}(t_{m+1}).
\]

For $t_k\neq0$, this is the factorization chart $T_{\mathbf i}$ in the higher-dimensional cell. For $t_k=0$, the factor $x_{i_k}(0)$ is the identity, and the map restricts to the factorization chart $T_{\mathbf j}$ in the boundary cell. The map is injective: its restrictions to $t_k=0$ and $t_k\neq0$ are injective factorization maps, and their images lie in two
disjoint double Bruhat cells. Moreover, the source and the target
\[
G^{u',v'}\sqcup G^{u,v}
\]
are both real manifolds of the same dimension. Therefore, by invariance of domain,
$\widetilde{x}_{\mathbf i,\mathbf j}$ is an open embedding onto its image.
In particular,
\[
T_{\mathbf j}\sqcup T_{\mathbf i}
\]
is an open neighbourhood of $x$, which is disjoint from $\D(G^{u,v})$. This contradicts the assumption that $x$ lie in the closure of $\D(G^{u,v})$.
\end{proof}

\section{The \texorpdfstring{$\mathrm{SL}_2$}{SL2} and \texorpdfstring{$\mathrm{SL}_3$}{SL3} calculations}\label{sec:sl2sl3}

In this section, we will give calculation on all double Bruhat cells for group $\SL_2$ and $\SL_3$. We will give examples of ideas in section \ref{sec:symmetries} and section \ref{sec:tracking}, by comparing those explicit expression of deep loci on double Bruhat cells.

\subsection{The baby example: \texorpdfstring{$\SL_2$}{SL2}}

Let
\[
   M=\begin{pmatrix}a&b\\ c&d\end{pmatrix}\in \SL_2.
\]
The Weyl Group is $S_2=\langle e,s\rangle$, where $e$ denotes the identity, and $s$ be the nontrivial transposition $(2\; 1)$. For double Weyl Group $S_2\times S_2$, it has $4$ different elements: $(e,e), (e,s), (s,e), (s,s)$. Each of the leading three elements correspond to only one reduced word, whose coordinate torus covers all the double Bruhat cell, so there will be no deep points. 
The open double Bruhat cell $G^{s,s}$ is cut out by
\[
   b\neq0,
   \qquad
   c\neq0.
\]
There are two double reduced words, $\bar11$ and $1\bar1$.  Their mutable chamber variables may be chosen to be $a$ and $d$, respectively.  Thus the two seed tori are
\[
   T_{\bar1,1}=\{a b c\neq0\},
   \qquad
   T_{1,\bar1}=\{d b c\neq0\}.
\]
Inside $G^{s,s}$, the deep locus for the reduced-word atlas is therefore
\[
   \{a=0\}\cap\{d=0\}\cap\{bc\neq0\}\cap\{ad-bc=1\}.
\]
Since $a=d=0$ and $ad-bc=1$ imply $bc=-1$, we get
\[
   \cD(G^{s,s})=
   \left\{
   \begin{pmatrix}
   0&b\\ c&0
   \end{pmatrix}:bc=-1
   \right\}.
\]
as we computed thoroughly in the first section.

\subsection[Computations in SL3]{Computations in \texorpdfstring{$\SL_3$}{SL3}}
\medskip

For the case that $G=\SL_3$, the Weyl group is $$S_3=\{e, s_1, s_2, s_2s_1, s_1s_2, w_0=s_1s_2s_1=s_2s_1s_2\}$$. We use the notation of matrix indices of element $a\in G$ from:
\[
a=\begin{pmatrix}
a_{11} &a_{12} &a_{13}\\
a_{21} &a_{22} &a_{23}\\
a_{31} &a_{32} &a_{33}
\end{pmatrix}
\]

and $a_{ij}$ should be of non-zero value unless an upper $\star$ is attached (\textit{e.g.}  $ a_{12}^\star$). In the following table, we show the number of irreducible components of the deep loci of all double Bruhat cells, and $0$ means that the deep locus is empty.
\[
\begin{array}{c|cccccc}
 u\backslash v & e&s_1&s_2&s_1s_2&s_2s_1&w_0\\ \hline
 e        &0&0&0&0&0&1\\
 s_1      &0&1&0&1&1&1\\
 s_2      &0&0&1&1&1&1\\
 s_1s_2   &0&1&1&2&2&2\\
 s_2s_1   &0&1&1&2&2&2\\
 w_0      &1&1&1&2&2&3
\end{array}
\]

We will give the results and a sketch of computation on every double Bruhat cell $G^{u,v}$ row by row..

\bigskip

\paragraph{\bf Row \(u=e\).} Elements $(e, v)$ have unique reduced-word presentation when $v\neq w_0$, and the corresponding double Bruhat cell is covered by the unique coordinate torus, which means that:
\[
\D_{e,v}=\varnothing\quad(v\neq w_0),
\]

For the case $u=w_0$, there are two different reduced words: $121$ and $212$. We already know in section 1.3 that the coordinate tori of them are respectively $\{x|x_{23}\neq 0\}$ and $\{x|x_{12}\neq 0\}$ for $x\in B$ Borel of $\SL_3$. As a result, the deep locus is: 
\[
\D_{e,w_0}
=\{a=\!\left(
\begin{array}{ccc}
a_{11}&0&a_{13}\\
0&a_{22}&0\\
0&0&a_{33}
\end{array}
\right)\!,\; \textup{det}(a)=1 \}
\]

\paragraph{\bf Row \(u=s_2\).} For $v=e$ and $v=s_1$:

\[
\D_{s_2,e}=\D_{s_2,s_1}=\varnothing,
\]

For $v=s_2$, The cell is the image of the open double Bruhat cell in $\GL_2$ under map:

\[
S(\GL_1 \times \GL_2) \to \SL_3, \; (a,b) \mapsto a\otimes \textup{id} + \textup{id}\otimes b
\]

so we have:
\[
\D_{s_2,s_2}
=\{a=\!\left(
\begin{array}{ccc}
a_{11}&0&0\\
0&0&a_{23}\\
0&a_{32}&0
\end{array}
\right)\!,\; a_{11}a_{23}a_{32}=-1\}
\]

The coordinate tori of case $v=s_1s_2$ and $v=s_2s_1$ are just left/right shift of the case $v=s_2$ by $x_1(t)$. Homeomorphisms

$$\C^* \times G^{s_2,s_2}\to G^{s_2,s_1s_2},\; (t,a) \mapsto x_1(t)a$$
$$\C^* \times G^{s_2,s_2}\to G^{s_2,s_1s_2},\; (t,a) \mapsto a x_1(t)$$
identify the whole cell as a shift by left/right multiplication on $G^{s_2,s_2}$, and so the deep locus. We have

\[
\D_{s_2,s_1s_2}
=\{a=\!\left(
\begin{array}{ccc}
a_{11}&0&a_{13}\\
0&0&a_{23}\\
0&a_{32}&0
\end{array}
\right)\!,\; a_{11}a_{23}a_{32}=-1\}
\]
\[
\D_{s_2,s_2s_1}
=\{a=\!\left(
\begin{array}{ccc}
a_{11}&a_{12}&0\\
0&0&a_{23}\\
0&a_{32}&0
\end{array}
\right)\!,\; a_{11}a_{23}a_{32}=-1\}
\]
\
For the last case, when $v=w_0$, the homeomorphism
$$\C^*\times \C^* \times G^{s_2,s_2}\to G^{s_2,s_1s_2},\; (t,s ,a) \mapsto x_1(t) a x_1(s)$$
identify the deep locus as a shift by both side on $\D(G^{s_2,s_2})$, so we have:

$$\C^*\times \C^* \times \D(G^{s_2,s_2})\to \D(G^{s_2,s_1s_2}),\; (t,s ,a) \mapsto x_1(t) a x_1(s)$$
and

\[
\D_{s_2,w_0}
=\{a=\!\left(
\begin{array}{ccc}
a_{11}&a_{12}&a_{13}\\
0&0&a_{23}\\
0&a_{32}&0
\end{array}
\right)\!,\; a_{11}a_{23}a_{32}=-1\}
\]

\paragraph{\bf Row \(u=s_1\).}
The case is similar with the story of row $u=s_2$.

Each cell in this case is related to one cell in $u=s_2$ case by the anti-diagonal transpose action. We have the deep loci:

\[
\D_{s_1,e}=\D_{s_1,s_2}=\varnothing,
\]
\[
\D_{s_1,s_1}
=\{\!\left(
\begin{array}{ccc}
0&a_{12}&0\\
a_{21}&0&0\\
0&0&a_{33}
\end{array}
\right)\!,\;a_{12}a_{21}a_{33}=-1\}
\]
\[
\D_{s_1,s_1s_2}
=\{\!\left(
\begin{array}{ccc}
0&a_{12}&a_{13}\\
a_{21}&0&0\\
0&0&a_{33}
\end{array}
\right)\!,\;a_{12}a_{21}a_{33}=-1\}
\]
\[
\D_{s_1,s_2s_1}
=\{\!\left(
\begin{array}{ccc}
0&a_{12}&0\\
a_{21}&0&a_{23}\\
0&0&a_{33}
\end{array}
\right)\!,\;a_{12}a_{21}a_{33}=-1\}
\]
\[
\D_{s_1,w_0}
=\{\!\left(
\begin{array}{ccc}
0&a_{12}&a_{13}\\
a_{21}&0&a_{23}\\
0&0&a_{33}
\end{array}
\right)\!,\;a_{12}a_{21}a_{33}=-1\}
\]

\paragraph{\bf Row \(u=s_1s_2\).}

The first three cases are related to the rows above by a transpose action. We directly get:
\[
\D_{s_1s_2,e}=\varnothing,
\]
\[
\D_{s_1s_2,s_2}
=\{\!\left(
\begin{array}{ccc}
a_{11}&0&0\\
a_{21}&0&a_{23}\\
0&a_{32}&0
\end{array}
\right)\!,\;a_{23}a_{32}a_{11}=-1\}
\]
\[
\D_{s_1s_2,s_1}
=\{\!\left(
\begin{array}{ccc}
0&a_{12}&0\\
a_{21}&0&0\\
0&a_{32}&a_{33}
\end{array}
\right)\!,\;a_{12}a_{21}a_{33}=-1\}
\]
For the case $v=s_2s_1$, by L-U decomposition, all matricees with non-zero leading or tailing principal minors are not deep, so we only need to check if the remaining matrices are coverd by some coordinate tori. The answer is: 

\[
\begin{aligned}
\D_{s_1s_2, s_2s_1}
={}&
\{a=\!\left(
\begin{array}{ccc}
a_{11}&a_{12}&0\\
a_{21}&a_{22}&a_{23}\\
0&a_{32}&0
\end{array}
\right)\!,\;a_{11}a_{22}-a_{12}a_{21}=0,\;a_{11}a_{23}a_{32}=-1\}\\
&\cup
\{a=\!\left(
\begin{array}{ccc}
0&a_{12}&0\\
a_{21}&a_{22}&a_{23}\\
0&a_{32}&a_{33}
\end{array}
\right)\!,\;a_{22}a_{33}-a_{23} a_{32}=0,\;a_{12}a_{21}a_{33}=-1\}.
\end{aligned}
\]
The next case also have two non-trivial irreducible components:
\[
\D_{s_1s_2,s_1s_2}
=\{a=\!
\begin{pmatrix}
a_{11}^\star&0&a_{13} \\
a_{21}&0&a_{23}^\star \\
0&a_{32}&0
\end{pmatrix}
\textup{ or }\!\left(
\begin{array}{ccc}
0&a_{12}^\star&a_{13}\\
a_{21}&0&0\\
0&a_{32}&a_{33}^\star
\end{array}
\right)| \textup{ det}(a)=1\}
\]

\[
\begin{aligned}
\D_{s_1s_2,w_0}
={}&
\{a=\!\left(
\begin{array}{ccc}
0&0&a_{13}\\
a_{21}&a_{22}&a_{23}^\star\\
0&a_{32}&0
\end{array}
\right)\!,\;a_{12}a_{32}a_{31}=1\}\\
&\cup
\{a=\!\left(
\begin{array}{ccc}
0&a_{12}^\star&a_{13}\\
a_{21}&a_{22}&0\\
0&a_{32}&0
\end{array}
\right)\!,\;a_{12}a_{32}a_{31}=1\}.
\end{aligned}
\]

\paragraph{\bf Row \(u=s_2s_1\).}
The first three cases here are the same with cases $u=e,s_1,s_2, v=s_2s_1$ that we discussed.

We directly write:
\[
\D_{s_2s_1,e}=\varnothing
\]

\[
\D_{s_2s_1,s_2}
=\{a=\!\left(
\begin{array}{ccc}
a_{11}&0&0\\
0&0&a_{23}\\
a_{31}&a_{32}&0
\end{array}
\right)\!,\;a_{11}a_{23}a_{32}=-1\},
\]
\[
\D_{s_2s_1,s_1}
=\{a=\!\left(
\begin{array}{ccc}
0&a_{12}&0\\
a_{21}&0&0\\
a_{31}&0&a_{33}
\end{array}
\right)\!,\;a_{12}a_{21}a_{33}=-1\}.
\]

For the case $v=s_1s_2$ and later, we compute also by discussing on cases where L-U decomposition fails:\[
\D_{s_2s_1,s_1s_2}
=\{a=\!
\begin{pmatrix}
a_{11}&0&a_{13} \\
0&0&a_{23} \\
a_{31}&a_{32}&0
\end{pmatrix}
\textup{ or }\!\left(
\begin{array}{ccc}
0&a_{12}&a_{13}\\
a_{21}&0&0\\
a_{31}&0&a_{33}
\end{array}
\right)| \textup{ det}(a)=1\}
\]

For \(v=s_2s_1\), transposing the two components of
\(\D_{s_1s_2,s_1s_2}\) gives the two components
\[
\begin{aligned}
\D_{s_2s_1,s_2s_1}
={}&
\{a=\!\left(
\begin{array}{ccc}
a_{11}^\star &a_{12}&0\\
0&0&a_{23}\\
a_{31}& a_{32}^\star &0
\end{array}
\right)\!,\;a_{12}a_{23}a_{31}=1\}\\
&\cup
\{a=\!\left(
\begin{array}{ccc}
0&a_{12}&0\\
a_{21}^\star &0&a_{23}\\
a_{31}&0&a_{33}^\star
\end{array}
\right)\!,\;a_{12}a_{23}a_{31}=1\}.
\end{aligned}
\]
For \(v=w_0\), if the matrix admits L-U decomposition, then the upper part must be in deep locus of the open dense cell of Borel subgroup. We may calculate that there be no deep points in this case. Then by discussing cases we cannot do L-U(and U-L) decomposition, we get
\[
\D_{s_2s_1,w_0}
=\{a=\!
\begin{pmatrix}
a_{11}^\star&a_{12}&a_{13} \\
0&0&a_{23} \\
a_{31}&0&0
\end{pmatrix}
\textup{ or }\!\left(
\begin{array}{ccc}
0&a_{12}&a_{13}\\
0&0&a_{23}\\
a_{31}&0&a_{33}^\star
\end{array}
\right)| \textup{ det}(a)=1\}
\]
\paragraph{\bf Row \(u=w_0\).}
The first five nonempty entries are obtained by transposing the corresponding
entries in the \(u=e,s_1,s_2,s_2s_1,s_1s_2\) rows.  We write those results directly by:

\[
\D_{w_0,e}
=\{a=\!\left(
\begin{array}{ccc}
a_{11}&0&0\\
0&a_{22}&0\\
a_{31}&0&a_{33}
\end{array}
\right)\!,\;a_{11}a_{22}a_{33}=1\},
\]
\[
\D_{w_0,s_1}
=\{a=\!\left(
\begin{array}{ccc}
0&a_{12}&0\\
a_{21}&0&0\\
a_{31}&a_{32}&a_{33}
\end{array}
\right)\!,\;a_{12}a_{21}a_{33}=-1\},
\]
\[
\D_{w_0,s_2}
=\{a=\!\left(
\begin{array}{ccc}
a_{11}&0&0\\
a_{21}&0&a_{23}\\
a_{31}&a_{32}&0
\end{array}
\right)\!,\;a_{11}a_{23}a_{32}=-1\},
\]

\[
\D_{w_0,s_1s_2}
=\{a=\!
\begin{pmatrix}
a_{11}^\star&0&a_{13} \\
a_{21}&0&0 \\
a_{31}&a_{32}&0
\end{pmatrix}
\textup{ or }\!\left(
\begin{array}{ccc}
0&0&a_{13}\\
a_{21}&0&0\\
a_{31}&a_{32}&a_{33}^\star
\end{array}
\right)| \textup{ det}(a)=1\}
\]

\[
\begin{aligned}
\D_{w_0,s_2s_1}
={}&
\{a=\!\left(
\begin{array}{ccc}
0&a_{12}&0\\
a_{21}^{\star}&a_{22}&a_{23}\\
a_{31}&0&0
\end{array}
\right)\!,\;a_{12}a_{23}a_{31}=1\}\\
&\cup
\{a=\!\left(
\begin{array}{ccc}
0&a_{12}&0\\
0&a_{22}&a_{23}\\
a_{31}&a_{32}^{\star}&0
\end{array}
\right)\!,\;a_{12}a_{23}a_{31}=1\}.
\end{aligned}
\]
Finally the open cell has three irredundant families:
\[
\begin{aligned}
\D_{w_0,w_0}
={}&\left\{
a=\!\begin{pmatrix}
0&0&a_{13}\\
a_{21}^{\star}&a_{22}&0\\
a_{31}&a_{32}^{\star}&0
\end{pmatrix}
:\ \det(a)=1,\;
\det\begin{pmatrix}a_{21}^{\star}&a_{22}\\ a_{31}&a_{32}^{\star}\end{pmatrix}\neq0
\right\}\\
&\cup\left\{
a=\!\begin{pmatrix}
0&a_{12}^{\star}&a_{13}\\
0&a_{22}&a_{23}^{\star}\\
a_{31}&0&0
\end{pmatrix}
:\ \det(a)=1,\;
\det\begin{pmatrix}a_{12}^{\star}&a_{13}\\ a_{22}&a_{23}^{\star}\end{pmatrix}\neq0
\right\}\\
&\cup\left\{
a=\!\begin{pmatrix}
a_{11}^{\star}&0&a_{13}\\
0&a_{22}&0\\
a_{31}&0&a_{33}^{\star}
\end{pmatrix}
:\ \det(a)=1,\;
\det\begin{pmatrix}a_{11}^{\star}&a_{13}\\ a_{31}&a_{33}^{\star}\end{pmatrix}\neq0
\right\}.
\end{aligned}
\]
Indeed, in these three cases the determinant equations are respectively
\[
a_{13}(a_{21}^{\star}a_{32}^{\star}-a_{22}a_{31})=1,\qquad
a_{31}(a_{12}^{\star}a_{23}^{\star}-a_{13}a_{22})=1,\qquad
a_{22}(a_{11}^{\star}a_{33}^{\star}-a_{13}a_{31})=1.
\]
The anti-diagonal support stratum
\[
\left\{
a=\!\left(
\begin{array}{ccc}
0&0&a_{13}\\
0&a_{22}&0\\
a_{31}&0&0
\end{array}
\right)\!,\;a_{13}a_{22}a_{31}=-1
\right\}
\]
is contained in all
displayed families above, by taking \(a_{11}^{\star}=a_{33}^{\star}=0\).

\subsection{Checking: the tracking formula}

We explain how the tracking formula from
Section~\ref{sec:tracking} can be applied in the example above. Recall that, for a terminal letter $\eta \in S_{u,v}$,
where $\eta=\bar i$ or $\eta=i$, we write
\begin{equation*}
   (u,v)\eta =
   \begin{cases}
      (us_i,v),& \eta=\bar i,\\
      (u,vs_i),& \eta=i.
   \end{cases}
\end{equation*}
The tracking formula says that
\begin{equation*}
   \D^{u,v}
   =
   \bigcap_{\eta\in S_{u,v}}
   \left(\D^0_\eta\cup \D^1_\eta\right),
   \qquad
   \D^1_\eta=\overline{ \D^{(u,v)\eta}\cdot x_\eta(\mathbb C^\ast)},
\end{equation*}
where $\D^0_\eta$ is the complement of the image of map:
\begin{equation*}
   G^{(u,v)\eta}\times x_\eta(\mathbb C^\ast)\longrightarrow G^{u,v}.
\end{equation*}
Thus each component of $\D^{u,v}$ should come out of a
deep component in a smaller cell, or as a failure of one of the terminal product
charts.

\paragraph{\bf Example: the cell $G^{s_2,w_0}$.}
The largest cell in the row $u=s_2$ gives a clean example where the tracking theorem shows the deep locus
comes entirely from a smaller deep locus. From the table,
\begin{equation*}
\D_{s_2,s_2}
=
\left\{
\begin{pmatrix}
a_{11}&0&0\\
0&0&a_{23}\\
0&a_{32}&0
\end{pmatrix}
:\ a_{11}a_{23}a_{32}=-1
\right\}.
\end{equation*}
Multiplying on the left and on the right by the two extra $x_1$-directions gives
\begin{equation*}
x_1(t)
\begin{pmatrix}
a_{11}&0&0\\
0&0&a_{23}\\
0&a_{32}&0
\end{pmatrix}
x_1(s)
=
\begin{pmatrix}
a_{11}&s a_{11}&t a_{23}\\
0&0&a_{23}\\
0&a_{32}&0
\end{pmatrix}.
\end{equation*}
this is exactly
\begin{equation*}
\D^{s_2,w_0}
=
\left\{
\begin{pmatrix}
a_{11}&a_{12}&a_{13}\\
0&0&a_{23}\\
0&a_{32}&0
\end{pmatrix}
:\ a_{11}a_{23}a_{32}=-1
\right\}.
\end{equation*}
Thus, in this case, the terminal $x_1$-directions do not create new deep
components: the unique component is
\begin{equation*}
   \D^{s_2,w_0}
   =
   x_1(\mathbb C^\ast)\,\D^{s_2,s_2}\,x_1(\mathbb C^\ast)
   =
   \D^{s_2,s_1s_2}\,x_1(\mathbb C^\ast).
\end{equation*}
which comes from purely $\D^1_1$ when $\eta=1$. The story keeps the same if we start using tracking theory from left.

\subsection{Checking: Symmetries}

We next check the two symmetries from Section~\ref{sec:symmetries} on the
largest examples in the table. For the inversion map we have check that it induces a bijection between double Bruhat cells correspond to words inverse to each other. For other two, Recall we use the notation
\begin{equation*}
   \tau(a)=a^T,
   \qquad
   \alpha(a)=Ja^TJ,
\end{equation*}
where $J$ is the anti-diagonal permutation matrix.

\paragraph{\bf Example 1: transpose symmetry between $G^{s_2s_1,w_0}$ and
$G^{w_0,s_1s_2}$.}
The transpose map sends
\begin{equation*}
   \tau:G^{u,v}\longrightarrow G^{v^{-1},u^{-1}}.
\end{equation*}
Thus
\begin{equation*}
   \tau:G^{s_2s_1,w_0}\longrightarrow G^{w_0,s_1s_2}.
\end{equation*}
The two components of $\D_{s_2s_1,w_0}$ are
\begin{equation*}
\left\{
\begin{pmatrix}
a_{11}^{\star}&a_{12}&a_{13}\\
0&0&a_{23}\\
a_{31}&0&0
\end{pmatrix}
\right\},
\qquad
\left\{
\begin{pmatrix}
0&a_{12}&a_{13}\\
0&0&a_{23}\\
a_{31}&0&a_{33}^{\star}
\end{pmatrix}
\right\},
\end{equation*}
with $\det(a)=1$. Transposing them gives
\begin{equation*}
\begin{pmatrix}
a_{11}^{\star}&0&a_{31}\\
a_{12}&0&0\\
a_{13}&a_{23}&0
\end{pmatrix},
\qquad
\begin{pmatrix}
0&0&a_{31}\\
a_{12}&0&0\\
a_{13}&a_{23}&a_{33}^{\star}
\end{pmatrix}.
\end{equation*}
These are precisely the two components listed
for $\D_{w_0,s_1s_2}$.

\bigskip

\paragraph{\bf Example 2: the open cell $G^{w_0,w_0}$.}
For the open cell, both $\tau$ and $\alpha$ preserve the cell. Let
\begin{equation*}
   \mathcal D^{w_0,w_0}=C_1\cup C_2\cup C_3
\end{equation*}
be the three components displayed above. The ordinary transpose exchanges the
first two components and fixes the third:
\begin{equation*}
   \tau(C_1)=C_2,\qquad
   \tau(C_2)=C_1,\qquad
   \tau(C_3)=C_3.
\end{equation*}
Indeed, $C_1$ has zero pattern
\begin{equation*}
\begin{pmatrix}
0&0&\ast\\
\ast&\ast&0\\
\ast&\ast&0
\end{pmatrix},
\end{equation*}
and its transpose has exactly the zero pattern of $C_2$, so they maps onto each other under transposing. The component $C_3$ is symmetric across
the main diagonal.

\bigskip

The anti-diagonal transpose fixes each component respectively as a relabeling of metrix indices,
\begin{equation*}
   \alpha(C_i)=C_i,\qquad i=1,2,3,
\end{equation*}
For instance,
\begin{equation*}
\alpha\!\left(
\begin{pmatrix}
0&0&a_{13}\\
a_{21}^{\star}&a_{22}&0\\
a_{31}&a_{32}^{\star}&0
\end{pmatrix}
\right)
=
\begin{pmatrix}
0&0&a_{13}\\
a_{32}^{\star}&a_{22}&0\\
a_{31}&a_{21}^{\star}&0
\end{pmatrix},
\end{equation*}
which keeps in $C_1$.

\subsection{Checking: Boundary-compatibility}

Finally we check boundary compatibility on codimension-one boundaries of the
largest cell $G^{w_0,w_0}$. The form to be checked is
\begin{equation*}
   G^{u,v}\cap \overline{\D^{u',v'}}
   \subseteq
   \D^{u,v},
   \qquad
   (u,v)<(u',v')
\end{equation*}
for the left/white weak Bruhat boundary. 

\paragraph{\bf Example 1: the boundary $G^{s_2s_1,w_0}\subset \overline{G^{w_0,w_0}}$.}
The boundary from $w_0$ to $s_2s_1$ on the left is cut out, inside the left
Bruhat closure, by
\begin{equation*}
   a_{21}a_{32}-a_{22}a_{31}=0
\end{equation*}
The open part of $G^{s_2s_1,w_0}$ also requires
\begin{equation*}
   a_{13}\neq0,\qquad
   a_{11}a_{32}-a_{12}a_{31}\neq 0
\end{equation*}
on the left side, together with the opposite open-cell conditions for $w_0$.

Now restrict the three open-cell components $C_1,C_2,C_3$ of
$\mathcal D^{w_0,w_0}$ to this boundary. On $C_1$, the determinant restriction says that this minor is nonzero, so
$C_1$ does not meet this boundary cell.

On $C_2$, the boundary equation becomes
\begin{equation*}
   a_{22}a_{31}=0.
\end{equation*}
Since $a_{31}\neq0$ on $C_2$, we get $a_{22}=0$. The limiting matrices are
therefore of the form
\begin{equation*}
\begin{pmatrix}
0&a_{12}&a_{13}\\
0&0&a_{23}\\
a_{31}&0&0
\end{pmatrix},
\qquad
a_{31}a_{12}^{\star}a_{23}^{\star}=1,
\end{equation*}
which is contained in the second component of $\D^{s_2s_1,w_0}$.

On $C_3$, we have $\Delta_{13,12}=0$, contradicts with the restriction conditions on $G^{s_2s_1,w_0}$.  Therefore
\begin{equation*}
   G^{s_2s_1,w_0}\cap \overline{\mathcal D^{w_0,w_0}}
   \subseteq
   \mathcal D^{s_2s_1,w_0}.
\end{equation*}
This gives an example how boundary compatibility works for double Bruhat cells.
\section{The cells \texorpdfstring{$G^{e,u}\subset SL_4$}{G(e,u) of SL4.}}\label{sec:sl4borel}

We now give another explicit calculation in type $A$.  In this section we compute the reduced-word deep loci for the Borel double Bruhat cells
\[
   G^{e,u}=B\cap B^-uB^-\subset \SL_4.
\]
The main new case is the open Borel cell
\[
   X=G^{e,w_0}=B\cap B^-w_0B^- .
\]

Since the diagonal torus does not affect vanishing of factorization variables or minor variables, we first work on the upper unipotent quotient and then multiply the answer by $H$.  We write
\[
U_4(a,b,c,d,e,f)=
\begin{pmatrix}
1&a&b&c\\
0&1&d&e\\
0&0&1&f\\
0&0&0&1
\end{pmatrix}.
\]
The frozen variables consist of variables reduced to $1$ and three other variables. For the open cell, they are three opposite boundary minors on the open cell, which are
\begin{equation*}\label{eq:sl4_borel_frozen_matrix}
   \Delta_{1,4}=c,
   \qquad
   \Delta_{12,34}=be-cd,
   \qquad
   \Delta_{123,234}=adf-ae-bf+c.
\end{equation*}
Thus the open-cell condition on the unipotent quotient is
\begin{equation*}\label{eq:sl4_open_condition_matrix}
   q=c(be-cd)(adf-ae-bf+c)\neq0.
\end{equation*}
We discuss the cells $G^{e,u}$ from smaller to larger Coxeter length of $u$ and then give a separate minor-coordinate calculation for $G^{e,w_0}$.
\subsection{Explicit presentation of the deep loci}

We list the nonempty deep loci of factorized variables by length.  Every displayed matrix is understood as a subset of $U$; or we can treat them as the quotient of $B$ by left action of $H$.

\subsubsection*{$\ell = 0,1,2,3$}

For $\ell(u)\le 3$, all deep loci are empty except for $u=s_1s_2s_1$ and $u=s_2s_3s_2$.  They give nontrivial deep loci as:
\[
\D_{e,s_1s_2s_1}=
\left\{
\begin{pmatrix}
1&0&u&0\\
0&1&0&0\\
0&0&1&0\\
0&0&0&1
\end{pmatrix}
:u\in\C^{\times}
\right\},
\qquad
\D_{e,s_2s_3s_2}=
\left\{
\begin{pmatrix}
1&0&0&0\\
0&1&0&u\\
0&0&1&0\\
0&0&0&1
\end{pmatrix}
:u\in\C^{\times}
\right\}.
\]

\subsubsection*{$\ell=4$}

Among the length-$4$ cells, only $2132$ has empty deep locus.  The four nonempty ones come from left/right multiplication of $s_1s_2s_1$ and $s_2s_3s_2$ by a simple shearing transformation.  They are
\[
\D_{e,s_1s_2s_3s_2}=
\left\{
\begin{pmatrix}
1&u&0&uv\\
0&1&0&v\\
0&0&1&0\\
0&0&0&1
\end{pmatrix}
:u,v\in\C^{\times}
\right\},
\;
\D_{e,s_2s_3s_2s_1}=
\left\{
\begin{pmatrix}
1&u&0&0\\
0&1&0&v\\
0&0&1&0\\
0&0&0&1
\end{pmatrix}
:u,v\in\C^{\times}
\right\},
\]
\[
\D_{e,s_2s_1s_2s_3}=
\left\{
\begin{pmatrix}
1&0&u&uv\\
0&1&0&0\\
0&0&1&v\\
0&0&0&1
\end{pmatrix}
:u,v\in\C^{\times}
\right\},
\;
\D_{e,s_3s_2s_1s_2}=
\left\{
\begin{pmatrix}
1&0&u&0\\
0&1&0&0\\
0&0&1&v\\
0&0&0&1
\end{pmatrix}
:u,v\in\C^{\times}
\right\}.
\]

For the remaining length-$4$ cell $s_2s_1s_3s_2$, the coordinate tori from word $"2132"$ and $"2321"$ are the same, however they donot cover the whole cell. 

The deep locus is :
\[
\D_{e,s_2s_1s_3s_2}=
\left\{
\begin{pmatrix}
1&0&u&0\\
0&1&w&v\\
0&0&1&0\\
0&0&0&1
\end{pmatrix}
:u,v,w\in\C^{\times}
\right\}.
\]

\subsubsection*{$\ell=5$}

For $u=s_2s_3s_2s_1s_2$, the deep locus has two components:
\[
\D_{e,s_2s_3s_2s_1s_2}=
\left\{
\begin{pmatrix}
1&0&u&0\\
0&1&v&vw\\
0&0&1&w\\
0&0&0&1
\end{pmatrix}
:u,v,w\in\C^{\times}
\right\}
\cup
\left\{
\begin{pmatrix}
1&u&uv&0\\
0&1&v&w\\
0&0&1&0\\
0&0&0&1
\end{pmatrix}
:u,v,w\in\C^{\times}
\right\}.
\]

For $u=s_2s_1s_2s_3s_2$, the two components are
\[
\D_{e,s_2s_1s_2s_3s_2}=
\left\{
\begin{pmatrix}
1&0&u&uw\\
0&1&v&0\\
0&0&1&w\\
0&0&0&1
\end{pmatrix}
:u,v,w\in\C^{\times}
\right\}
\cup
\left\{
\begin{pmatrix}
1&u&0&uw\\
0&1&v&w\\
0&0&1&0\\
0&0&0&1
\end{pmatrix}
:u,v,w\in\C^{\times}
\right\}.
\]

For $u=s_3s_1s_2s_3s_1$, the two components are
\[
\D_{e,s_3s_1s_2s_3s_1}=
\left\{
\begin{pmatrix}
1&x&0&z\neq xy\\
0&1&0&y\\
0&0&1&w\\
0&0&0&1
\end{pmatrix}
:z\in\C^{\times}
\right\}
\cup
\left\{
\begin{pmatrix}
1&x&y&z\neq yw\\
0&1&0&0\\
0&0&1&w\\
0&0&0&1
\end{pmatrix}
:z\in\C^{\times}
\right\}.
\]

\subsubsection*{$\ell=6$: The open cell}

We now record the open-cell result in the same matrix style.  The reduced-word deep locus in factorized variables is the union of four components:
\[
D_1=
\left\{
\begin{pmatrix}
1&0&0&u\\
0&1&v&t\\
0&0&1&0\\
0&0&0&1
\end{pmatrix}
:u,v\in\C^{\times}
\right\},
\;
D_2=
\left\{
\begin{pmatrix}
1&0&t&u\\
0&1&v&0\\
0&0&1&0\\
0&0&0&1
\end{pmatrix}
:u,v\in\C^{\times}
\right\},
\]
\[
D_3=
\left\{
\begin{pmatrix}
1&u&0&uvw\\
0&1&v&vw\\
0&0&1&w\\
0&0&0&1
\end{pmatrix}
:u,v,w\in\C^{\times}
\right\},
\;
D_4=
\left\{
\begin{pmatrix}
1&u&uv&uvw\\
0&1&v&0\\
0&0&1&w\\
0&0&0&1
\end{pmatrix}
:u,v,w\in\C^{\times}
\right\}.
\]
Equivalently,
\[
\D_{e,w_0}=D_1\cup D_2\cup D_3\cup D_4.
\]

These results also fit into the framwork of tracing theory. For example, the first component $D_1$ of deep locus of the open cell comes out of $\D_{1}^{0} \cap \D_{2}^{1} \cap \D_{3}^{0}$. Moreover, the four components of the open-cell deep locus are related by the involution $g\mapsto w_0g^Tw_0$, this is how symmetry act on the Borel case.

\subsection{\texorpdfstring{Minor variables' deep locus on the Borel open cell}{Minor variables' deep locus on the Borel open cell}}

Now we go back to the minor variables. For the open cell $G^{e,w_0}$, the Minors that appear in reduced-word charts are the following eight functions:
\begin{align*}
A&=\Delta_{1,2}=a,
&
B&=\Delta_{1,3}=b,
&
C&=\Delta_{12,23}=ad-b,
&
D&=\Delta_{12,13}=d,\\
E&=\Delta_{123,134}=df-e,
&
F&=\Delta_{12,14}=e,
&
G&=\Delta_{123,124}=f,
&
H&=\Delta_{12,24}=ae-c.
\end{align*}
These minors are of type $\Delta_{[n],\omega[n]}$, where $\omega \in S_4$ acts naturally on the set $[n]$.  We now compute the open-cell deep locus directly from the eight reduced-word minor charts in \eqref{eq:sl4_reduced_word_triangulations}.  The point is very simple:

\begin{quote}
A matrix lies in the deep locus if and only if in every one of the eight triangulations in \eqref{eq:sl4_reduced_word_triangulations}, at least one of the three displayed minors is zero.
\end{quote}

So we mark some of the variables $A,B,C,D,E,F,G,H$ by zero, and ask that every one of the eight triples
\[
ABC,\ BCD,\ CDE,\ DEF,\ EFG,\ FGH,\ AGH,\ ACH
\]
contains at least one marked variable.  The minimal ways to do this and still stay inside the open cell are precisely
\begin{equation*}
(A,D,G),\qquad (A,D,F),\qquad (A,B,F),\qquad (B,E,H),\qquad (B,F,H).
\end{equation*}
\begin{equation*}
(C,E,G),\qquad (C,D,G),\qquad (C,E,H),\qquad (A,B,E,G),\qquad (C,D,F,H).
\end{equation*}
To translate them back to matrices, we use:
\[
A=a,\quad B=b,\quad C=ad-b,\quad D=d,\quad E=df-e,\quad F=e,\quad G=f,\quad H=ae-c.
\]
On the open cell the following three frozen minors variable must stay nonzero:
\begin{equation*}\label{eq:sl4_open_minor_conditions}
\Delta_{1,4}=c\neq0,
\qquad
\Delta_{13,24}=be-cd\neq0,
\qquad
\Delta_{123,234}=adf-ae-bf+c\neq0.
\end{equation*}

Then the ways $(A,D,F)$, $(B,F,H)$, $(C,D,G)$, $(C,E,H)$ and $(C,D,F,H)$ are impossible because they force one of the three frozen minors to vanish.  The remaining four cases give the four components of the open-cell deep locus in minor variables.
\bigskip

\paragraph{\bf Pattern $(A,B,F)=(0,0,0)$.}
This means
\[
a=0,\qquad b=0,\qquad e=0.
\]
Then the nonvanishing conditions reduce to
\[
c\neq0,\qquad d\neq0.
\]
Thus we obtain
\[
\left\{
\begin{pmatrix}
1&0&0&u\\
0&1&v&0\\
0&0&1&t\\
0&0&0&1
\end{pmatrix}
:u,v\in\C^{\times},\ t\in\C
\right\}=D_1'.
\]

\paragraph{\bf Pattern $(A,D,G)=(0,0,0)$.}
This means
\[
a=0,\qquad d=0,\qquad f=0.
\]
Then \eqref{eq:sl4_open_minor_conditions} become
\[
b\neq0,\qquad c\neq0,\qquad e\neq0,
\]
so we get
\[
\left\{
\begin{pmatrix}
1&0&u&v\\
0&1&0&w\\
0&0&1&0\\
0&0&0&1
\end{pmatrix}
:u,v,w\in\C^{\times}
\right\}=D_2'.
\]

\paragraph{\bf Pattern $(B,E,H)=(0,0,0)$.}
This means
\[
b=0,\qquad df-e=0,\qquad ae-c=0.
\]
Hence
\[
e=df,\qquad c=ae=adf.
\]
The open-cell conditions then force
\[
a\neq0,\qquad d\neq0,\qquad f\neq0,
\]
so we obtain
\[
\left\{
\begin{pmatrix}
1&u&0&uvw\\
0&1&v&vw\\
0&0&1&w\\
0&0&0&1
\end{pmatrix}
:u,v,w\in\C^{\times}
\right\}=D_3'.
\]

\paragraph{\bf Pattern $(C,E,G)=(0,0,0)$.}
This means
\[
ad-b=0,\qquad df-e=0,\qquad f=0.
\]
Therefore
\[
b=ad,\qquad e=0.
\]
The open-cell conditions reduce to
\[
c\neq0,\qquad d\neq0,
\]
and hence
\[
\left\{
\begin{pmatrix}
1&t&tv&u\\
0&1&v&0\\
0&0&1&0\\
0&0&0&1
\end{pmatrix}
:u,v\in\C^{\times},\ t\in\C
\right\}=D_4'.
\]

Combining the four cases gives the following theorem.

\begin{theorem}[Open-cell deep locus in minor variables]\label{thm:sl4_minor_deep}
For the open Borel cell $G^{e,w_0}\subset SL_4$, the deep locus obtained from the reduced-word minor charts is
\[
\D^{\mathrm{minor}}_{e,w_0}=D_1'\cup D_2'\cup D_3'\cup D_4',
\]
where $D_1',D_2',D_3',D_4'$ are the four matrix families displayed above.
\end{theorem}

\subsection{Comparison with the factorized-variable description}

Theorem~\ref{thm:sl4_minor_deep} gives exactly the same four components as the factorized-variable calculation stated earlier.  Thus, for the open cell, the two languages describe the same geometry:
\[
\D^{\mathrm{fact}}_{e,w_0}\simeq\D^{\mathrm{minor}}_{e,w_0}=D_1'\cup D_2'\cup D_3'\cup D_4'.
\]

The only difference is the viewpoint.
\begin{itemize}
\item In factorized variables, one asks which points fail to lie in every reduced-word factorization chart.
\item In minor variables, one asks which points make every reduced-word triangulation in \eqref{eq:sl4_reduced_word_triangulations} contain at least one zero minor.
\end{itemize}
Because the minors are explicit functions of the matrix entries, the second description translates back to the same four matrix families.

This is the cleanest way to compare the two descriptions of the deep locus: the factorized atlas and the minor atlas give different coordinates, but they are the same thing under the twist map.

\subsection{Mutation and full cluster chart: the \texorpdfstring{$A_3$}{A3} cluster algebra structure}
From \cite{BFZ05}, we know that the cluster algebra structure on $G^{e,w_0}$ is of finite type $A_3$, and it should have $9$ different cluster variables.  The eight cluster variables that appear in the reduced-word charts are $A,B,C,D,E,F,G,H$ above. A ninth cluster variable will also appear in the cluster structure, which is 
\[
K=\Delta_{13,34}=bf-c.
\]

However, $K$ does not appear in any of the eight reduced-word charts.  It is needed to complete the $A_3$ picture, but it is not a cluster variable in the BFZ reduced-word atlas.

\medskip
\noindent
\textbf{Hexagon model.}
We will realize the $A_3$ cluster algebra structure of $G^{e,w_0}$ on a \emph{Hexagon model}. Take a convex hexagon with vertices labeled $1,2,3,4,5,6$.  The nine diagonals of the hexagon are put in one-to-one correspondence with the nine quantities above by
\[
A\leftrightarrow (1,3),\quad
B\leftrightarrow (1,4),\quad
C\leftrightarrow (1,5),\quad
D\leftrightarrow (2,4),\quad
E\leftrightarrow (2,5),
\]
\[
F\leftrightarrow (2,6),\quad
G\leftrightarrow (3,5),\quad
H\leftrightarrow (3,6),\quad
K\leftrightarrow (4,6).
\]
A cluster chart is then represented by a triangulation of the hexagon, or a set of three non-crossing diagonals.  Replacing one diagonal of a triangulation by the other diagonal of the same quadrilateral is called a \emph{flip}, which corresponds exactly to the $A_3$ mutation move.

Among the $14$ triangulations of the hexagon, the reduced-word atlas on $G^{e,w_0}$ produces the following eight:
\begin{equation*}\label{eq:sl4_reduced_word_triangulations}
ABC,\quad BCD,\quad CDE,\quad DEF,\quad EFG,\quad FGH,\quad AGH,\quad ACH.
\end{equation*}
In other words, the reduced-word charts see only eight triangulations. Except for $CEG$, the ninth diagonal $K$ also gives 5 different hidden triangulations:
$$ABK,\quad AHK,\quad DFK,\quad BDK,\quad FHK$$
none of them appear in the reduced-word atlas.

\begin{figure}[ht]
\centering
\begin{tikzpicture}[scale=1.05, every node/.style={font=\small}]
\coordinate (1) at (150:2.3);
\coordinate (2) at (90:2.3);
\coordinate (3) at (30:2.3);
\coordinate (4) at (-30:2.3);
\coordinate (5) at (-90:2.3);
\coordinate (6) at (-150:2.3);
\draw (1)--(2)--(3)--(4)--(5)--(6)--cycle;
\foreach \i in {1,...,6}{\fill (\i) circle (1.2pt);}
\node[left] at (1) {$1$};
\node[above] at (2) {$2$};
\node[right] at (3) {$3$};
\node[right] at (4) {$4$};
\node[below] at (5) {$5$};
\node[left] at (6) {$6$};
\draw (1)--(3) node[midway, above left] {$A$};
\draw (1)--(4) node[midway, above] {$B$};
\draw (1)--(5) node[midway, left] {$C$};
\draw (2)--(4) node[midway, above right] {$D$};
\draw (2)--(5) node[midway, right] {$E$};
\draw (2)--(6) node[midway, left] {$F$};
\draw (3)--(5) node[midway, right] {$G$};
\draw (3)--(6) node[midway, left] {$H$};
\draw (4)--(6) node[midway, below] {$K$};
\end{tikzpicture}
\caption{The nine minor variables as the nine diagonals of a hexagon.}
\label{fig:sl4_hexagon_variables}
\end{figure}
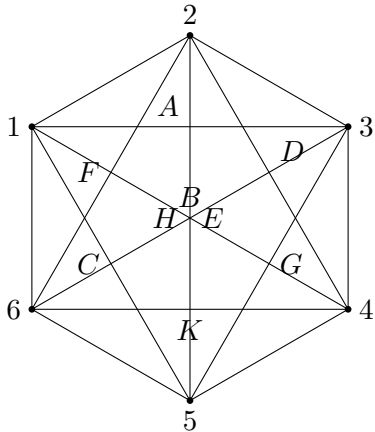

\begin{figure}[ht]
\centering
\begin{tikzpicture}[scale=0.95, every node/.style={font=\small}]
\begin{scope}[xshift=-4.7cm]
\coordinate (1) at (150:2.0);
\coordinate (2) at (90:2.0);
\coordinate (3) at (30:2.0);
\coordinate (4) at (-30:2.0);
\coordinate (5) at (-90:2.0);
\coordinate (6) at (-150:2.0);
\draw (1)--(2)--(3)--(4)--(5)--(6)--cycle;
\foreach \i in {1,...,6}{\fill (\i) circle (1pt);}
\draw[thick] (1)--(3) node[midway, above left] {$A$};
\draw[thick] (1)--(4) node[midway, above] {$B$};
\draw[thick] (1)--(5) node[midway, left] {$C$};
\node at (0,-2.8) {$ABC$};
\end{scope}
\draw[->, thick] (-0.9,0) -- (0.9,0);
\node at (0,0.45) {flip};
\begin{scope}[xshift=4.7cm]
\coordinate (1) at (150:2.0);
\coordinate (2) at (90:2.0);
\coordinate (3) at (30:2.0);
\coordinate (4) at (-30:2.0);
\coordinate (5) at (-90:2.0);
\coordinate (6) at (-150:2.0);
\draw (1)--(2)--(3)--(4)--(5)--(6)--cycle;
\foreach \i in {1,...,6}{\fill (\i) circle (1pt);}
\draw[thick] (1)--(4) node[midway, above] {$B$};
\draw[thick] (1)--(5) node[midway, left] {$C$};
\draw[thick] (2)--(4) node[midway, above right] {$D$};
\node at (0,-2.8) {$BCD$};
\end{scope}
\end{tikzpicture}
\caption{One elementary $A_3$ flip: the triangulation $ABC$ changes to $BCD$ by replacing $A=(1,3)$ with $D=(2,4)$. They correspond to word $"123121"$ and $"123212"$ respectively.}
\label{fig:sl4_hexagon_flip}
\end{figure}
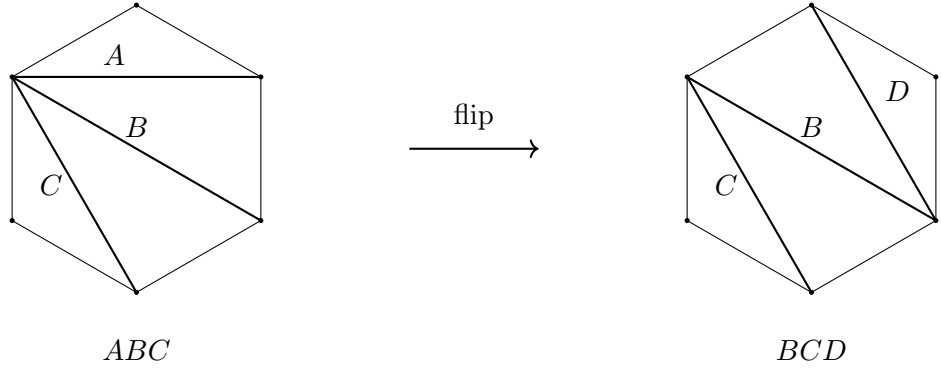

The deep locus of full cluster charts $\cD$ is then strictly smaller than the reduced word deep locus in this case. For example, a point in $D_1'-D_4'$ is deep in the reduced-word sense, but it is not deep in the full sense. It lies in the chart $DFK$ for: 
\[
d=v \neq 0,\qquad f=t \neq 0,\qquad bf-c=-u \neq 0,
\]

The only remaining elements in full deep locus are elements in $D_1'\cap D_4'$. This is our first example of the difference between the reduced-word deep locus and the full deep locus.  The reduced-word deep locus is more accessible and has more explicit structure, and the full deep locus is more intrinsic. 

\section{Conclusion and open problems}
\label{sec:conclusion-open}

We have seen that deep loci give a family of natural interesting subvarieties in special linear groups.  The definition of the deep locus is global: it's the set of deep points which escape every coordinate torus of the BFZ reduced-word atlas. However, we still cannot form a good description for the global structure of these deep loci. A useful slogan is that the two elementary moves
\[
   \bar i i \longleftrightarrow i\bar i,
   \qquad
   i(i+1)i \longleftrightarrow (i+1)i(i+1)
\]
correspond to $2$ most elementary cases: the deep loci of $\SL_2^{w_0,w_0}$ and $\SL_3^{e,w_0}$. We conclude by formulating several questions and conjectures which organize these observations.

\subsection{Deep loci near the identity}

A first question is to understand the deep loci close to the identity element.  In the simplest open double Bruhat cell  \(\SL_2^{w_0,w_0}\subset \SL_2\), the deep locus $\D(\SL_2^{w_0,w_0})$ does not contain the identity, but \(\D(\SL_3^{w_0,w_0})\) contains a irreducible component whose closure pass through the identity (and thus the whole Cartan subgroup).  

We will consider this problem on the infinitesimal germ
\[
   \widehat{\overline{\mathcal D^{e,w_0}}}_{\,I}
   \subset
   \widehat{U}_{\,I},
\]
where
\[
   U=
   \left\{
   I+\sum_{i<j} a_{ij}E_{ij}
   \right\}
\]
is the upper unipotent subgroup.

The basic example is \(SL_3^{e,w_0}\).  Writing
\[
   X=
   \begin{pmatrix}
   1&x_{12}&x_{13}\\
   0&1&x_{23}\\
   0&0&1
   \end{pmatrix},
\]
the two reduced words \(121\) and \(212\) give the two charts related by the braid move.  Take limit $x_{ij}x_{kl} \to 0$ (infinitesimal condition). The deep locus is then:
\[
   x_{12}=0,
   \qquad
   x_{23}=0,
   \qquad
   x_{13}\neq 0.
\]
Its closure passes through the identity along a non-simple-root direction \(E_{13}\), a commutator direction created by simple-root directions under the relation
\[
   [E_{12},E_{23}]=E_{13}.
\]

This suggests the following principle.

\begin{conjecture}[Identity-neighbourhood conjecture: Borel version]
\label{conj:identity-near-deep}
For \(G=\SL_N\), the infinitesimal germ of the deep locus of factorized variables near the identity is generated by subspaces generated by some root-directions.  More precisely, every irreducible component of
\[
   \widehat{\overline{\mathcal D^{e,w_0}}}_{\,I}
\]
is obtained by a set of roots $\{\alpha_i\}$.  \end{conjecture}

For example, the \(SL_3\) deep locus is generated by the single non-simple root \(\alpha_{13}\).  The \(SL_4\) deep locus has eight irreducible components in the formal germ, generated by the following sets of roots:
\[(\alpha_{12}, \alpha_{23}, \alpha_{34})*;\;(\alpha_{14}, \alpha_{23}, \alpha_{24});\; (\alpha_{14}, \alpha_{23}, \alpha_{13}),\]
\[(\alpha_{12}, \alpha_{13}, \alpha_{34});\;(\alpha_{12}, \alpha_{24}, \alpha_{34});\;(\alpha_{12}, \alpha_{23}, \alpha_{24})*\]
\[(\alpha_{13}, \alpha_{23}, \alpha_{34})*;\;(\alpha_{13}, \alpha_{23}, \alpha_{24}).\]

where $*$ means there are two irreducible components with the same infinitesimal structure. This is only an experimental observation. We can further study at least the infinitesimal structure of the deep locus, by trying to compute the tangent spaces of the deep locus nearby the identity, and maybe with multiplicity.

\subsection{The anti-diagonal stratum}

The anti-diagonal locus plays a special role in the largest open double Bruhat cell
\[
   G^{w_0,w_0}\subset \SL_N.
\]
Let \(J\) be an anti-diagonal matrix with indices $1$ or $-1$, we define the anti-diagonal torus
\[
   A_N^\circ
   :=
   \left\{
   hJ\in \SL_N:
   h\in H
   \right\}.
\]
Equivalently, \(A_N^\circ\) consists of matrices supported exactly on the anti-diagonal, with determinant one.

In the \(SL_2\) calculation, \(A_2^\circ\) is precisely the deep locus of the open cell.  In the \(SL_3\) calculation, the anti-diagonal support stratum is contained in the intersection of all deep components of \(G^{w_0,w_0}\).  This strongly suggests that the anti-diagonal stratum appearing as the deepest locus is not an accident of low rank, but a universal phenomenon.

\begin{proposition}[Anti-diagonal points are generically deep]
\label{prop:anti-diagonal-deep}
For the reduced-word BFZ atlas on the largest open cell \(G^{w_0,w_0}\subset \SL_N\), the anti-diagonal torus \(A_N^\circ\) is contained in the deep locus:
\[
   A_N^\circ\subset \mathcal D^{w_0,w_0}.
\]
\end{proposition}

\begin{proof}[Sketch]
We use the tracking description by terminal letters.  For every admissible terminal letter \(\eta\in S_{w_0,w_0}\), consider the terminal product map
\[
   \Phi_\eta:
   G^{(w_0,w_0)\eta}\times x_\eta(\mathbb C^*)
   \longrightarrow
   G^{w_0,w_0}.
\]
The image of \(\Phi_\eta\) is the part of the open cell covered by reduced words ending in \(\eta\).  The tracking formula says that a point escapes all coordinate tori ending in \(\eta\) if it lies in
\[
   \D^0_\eta
   =
   G^{w_0,w_0}- \operatorname{Im}(\Phi_\eta),
\]
or if it comes from the deep locus of the smaller cell.

Now take \(X\in A_N^\circ\).  Since \(X\) is supported only on the anti-diagonal, each row and each column contains exactly one nonzero entry.  Multiplication by a nontrivial elementary factor \(x_\eta(t)\), with \(t\neq0\), necessarily creates more nonzero entry in one of the adjacent rows or columns, and it won't decrease the rank pattern of northeast or southwest submatrices. Thus \(X\) does not lie in the image of \(\Phi_\eta\), and it is contained in \(\D^0_\eta\).

So far we have:
\[
   A_N^\circ\subset \D^0_\eta
   \qquad
   \text{for every }\eta\in S_{w_0,w_0}.
\]
Intersecting over all terminal letters gives
\[
   A_N^\circ
   \subset
   \bigcap_{\eta\in S_{w_0,w_0}}\D^0_\eta
   \subset
   \bigcap_{\eta\in S_{w_0,w_0}}
   \left(\D^0_\eta\cup \D^1_\eta\right)
   =
   \D_{w_0,w_0}.
\]
This proves the claim.  A full expanded proof should write nothing more but the rank obstruction for each terminal letter explicitly.
\end{proof}

This proposition leaves open another question: is the anti-diagonal torus itself an irreducible component, or is it always contained in larger deep components?  The \(SL_3\) calculation suggests the latter: the anti-diagonal support stratum is contained in several larger components, and we conjecture that in higher rank, the anti-diagonal may be better understood as a universal intersection core of the largest-cell deep locus.

\begin{conjecture}[Anti-diagonal core conjectures]
\label{conj:anti-diagonal-core}
For \(G=\SL_N\), the anti-diagonal torus \(A_N^\circ\) is contained in every irreducible components of \(\mathcal D^{w_0,w_0}\), or the ones that are invariant under anti-diagonal transpose. 

More generally, \(A_N^\circ\) should be exactly a universal intersection core, which means its is the intersection of all irreducible components containing it.
\end{conjecture}

\subsection{Final perspective}

In this paper we came up with the idea of terminal-letter tracking, which transports deep components from smaller to bigger double Bruhat cells.  However,still there is a lot to do on understanding these varieties, and their relationship with the cluster structure, Lie algebra, or Coxeter groups. We hoope this paper can be a starting point.

We summarize some of the main points and open problems at last:

\begin{enumerate}
\item Coxeter moves: the mixed 2-move \(1\bar{1}\leftrightarrow \bar{1}1\) and the braid move \(121\leftrightarrow 212\);
\item Core: the anti-diagonal stratum in the largest open cell.
\item Infinitesimal structure: the tangent space of the deep locus at the identity is generated by root directions, and we may discuss those irreducible components as linear subspaces with multiplicity.

\end{enumerate}

\appendix

\end{document}